\pdfoutput=1
\RequirePackage{ifpdf}
\ifpdf 
\documentclass[pdftex]{sigma}
\else
\documentclass{sigma}
\fi

\usepackage{tikz-cd}
\usepackage{tensor}

\numberwithin{equation}{section}

\newtheorem{Theorem}{Theorem}[section]
\newtheorem{Corollary}[Theorem]{Corollary}
\newtheorem{Lemma}[Theorem]{Lemma}
\newtheorem{Proposition}[Theorem]{Proposition}
{ \theoremstyle{definition}
\newtheorem{Definition}[Theorem]{Definition}
\newtheorem{Notation}[Theorem]{Notation}
\newtheorem{Example}[Theorem]{Example}
\newtheorem{Remark}[Theorem]{Remark} }

\begin{document}

\allowdisplaybreaks

\newcommand{\arXivNumber}{1606.06120}

\renewcommand{\PaperNumber}{007}

\FirstPageHeading

\ShortArticleName{Connected Lie Groupoids are Internally Connected}

\ArticleName{Connected Lie Groupoids are Internally Connected\\ and Integral Complete in Synthetic Dif\/ferential\\ Geometry}

\Author{Matthew BURKE}

\AuthorNameForHeading{M.~Burke}

\Address{4 River Court, Ferry Lane, Cambridge CB4 1NU, UK}
\Email{\href{mailto:matthew.burke@cantab.net}{matthew.burke@cantab.net}}
\URLaddress{\url{http://www.mwpb.uk}}

\ArticleDates{Received June 29, 2016, in f\/inal form January 13, 2017; Published online January 24, 2017}

\Abstract{We extend some fundamental def\/initions and constructions in the established generalisation of Lie theory involving Lie groupoids by reformulating them in terms of groupoids internal to a well-adapted model of synthetic dif\/ferential geometry. In particular we def\/ine internal counterparts of the def\/initions of source path and source simply connected groupoid and the integration of $A$-paths. The main results of this paper show that if a classical Hausdorf\/f Lie groupoid satisf\/ies one of the classical connectedness conditions it also satisf\/ies its internal counterpart.}


\Classification{22E60; 22E65; 03F55; 18B25; 18B40}

\section{Introduction}

In classical Lie theory we use a formal group law to represent the analytic approximation of a~Lie group. Recall that an \emph{$n$-dimensional formal group law $F$} is an $n$-tuple of power series in the variables $X_{1},\dots ,X_{n}$; $Y_{1},\dots ,Y_{n}$ with coef\/f\/icients in $\mathbb{R}$ such that the equalities
\begin{gather*}
F\big(\vec{X},\vec{0}\big)=\vec{X},\qquad F\big(\vec{0},\vec{Y}\big)=\vec{Y}\qquad \text{and} \qquad F\big(F\big(\vec{X},\vec{Y}\big),\vec{Z}\big)=F\big(\vec{X},F\big(\vec{Y},\vec{Z}\big)\big)
\end{gather*}
hold. In fact there is an equivalence of categories
\begin{equation}\label{classical-lie}
\begin{split}&\includegraphics{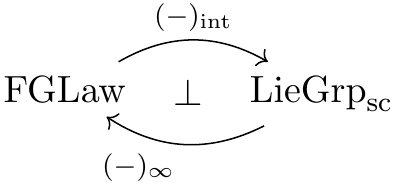}\end{split}
\end{equation}
between the category ${\rm FGLaw}$ of formal group laws and ${\rm LieGrp}_{\rm sc}$ of simply connected Lie groups. The functor $(-)_{\infty}$ is obtained by considering the Taylor expansion of the multiplication of the Lie group close to the identity element. For more details see Theorem~3 of Section~V.6 and Theorem~2 of Section~V.8 of Part~2 in~\cite{MR2179691}. The functor $(-)_{\rm int}$ therefore extends local data to global data and involves f\/inding solutions to smooth time-dependent left-invariant vector f\/ields.

In \cite{burke-relative-to-local-models} we generalise the underlying adjunction of \eqref{classical-lie} to an adjunction
\begin{equation}\label{new-lie-adjunction}
\begin{split}&\includegraphics{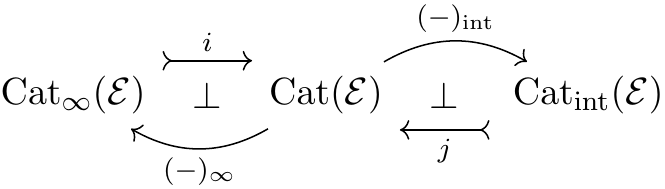}\end{split}
 \end{equation}
between two subcategories of the category of internal categories in a topos~$\mathcal{E}$. When considered together, this paper and~\cite{burke-relative-to-local-models} not only provide a more concise exposition of the thesis~\cite{BurkeThesis} but also contain several stronger results. The main improvement over~\cite{BurkeThesis} contained in this paper is the extension of the results about internal connectedness conditions from groups to groupoids.

In \cite{burke-relative-to-local-models} we prove the result analogous to Lie's second theorem in this context: when we apply the appropriate connectedness conditions (described in Section~\ref{sec:enriched-connectedness}) the functor $(-)_{\infty}j$ is full and faithful. In this paper we justify the work in~\cite{burke-relative-to-local-models} by describing the relationship between~\eqref{classical-lie} and~\eqref{new-lie-adjunction} in the case that $\mathcal{E}$ is a well-adapted model of synthetic dif\/ferential geometry (see Section~\ref{sec:sdg}). This is carried out in Section~\ref{sec:formal-group-laws} where we show that when we restrict ${\rm Cat}(\mathcal{E})$ to the full subcategory on the classical Lie groups, the functor $(-)_{\infty}$ coincides with the formal group law construction described in the Introduction of~\cite{MR506881}.

In addition we relate the adjunction \eqref{new-lie-adjunction} to the established generalisation of Lie theory involving Lie algebroids and Lie groupoids. (See for instance~\cite{MR2157566}.) A~\emph{Lie groupoid} is a groupoid in the category of smooth paracompact manifolds such that the source and target maps are submersions. A~\emph{Lie algebroid} is a vector bundle $A \rightarrow M$ together with a bundle homomorphism $\rho\colon A \rightarrow TM$ such that the space of sections $\Gamma(A)$ is a Lie algebra satisfying the following Leibniz law: for all $X,Y\in \Gamma(A)$ and $f\in C^{\infty}(M)$ the equality
\begin{gather*}[X,fY]=\rho(X)(f)\cdot Y+f[X,Y]\end{gather*}
holds.
In the theory of Lie groupoids and Lie algebroids we have have a functor
\begin{equation*}
\begin{tikzcd}
 {\rm LieAlgd} & {\rm LieGpd}, \lar[swap]{(-)_{\infty}}
\end{tikzcd}
\end{equation*}
which is full and faithful but not essentially surjective. Any Lie algebroid integrates to a topological groupoid, its Weinstein groupoid \cite{MR1973056}, but there can be obstructions to putting a smooth Hausdorf\/f structure on it. For instance see \cite{MR0197622} for a Lie algebroid whose Weinstein groupoid is a smooth but non-Hausdorf\/f Lie groupoid and~\cite{MR778785} for a~Lie algebroid whose Weinstein groupoid is non-smooth. Therefore when dealing with integrability (for instance in~\cite{MR1973056}) the category of smooth manifolds is enlarged to include non-Hausdorf\/f manifolds. Furthermore in~\cite{2006HsianHuaTsengandChenchangZhu22} Tseng and Zhu show that the category of dif\/ferentiable stacks contains all Weinstein groupoids whilst still retaining the concept of tangent vectors. Another approach, pursued in \cite{burke-relative-to-local-models}, is to use the theory of synthetic dif\/ferential geometry where the Weinstein groupoid construction is always possible.

In the process of reformulating the theory of Lie groupoids and Lie algebroids in~\cite{burke-relative-to-local-models} it is necessary to use internal versions of certain conditions describing connectedness and solutions to a specif\/ic type of vector f\/ield. In this paper we will justify these assumptions by showing that all classical Hausdorf\/f Lie groupoids satisfy these stronger conditions. Since the Weinstein groupoid construction is always possible in $\mathcal{E}$ the assumption that our groupoids are Hausdorf\/f does not af\/fect the part of the theory dealing with integrability, only the extent to which the conditions involving completeness and solutions to vector f\/ields generalise the classical ones. So unless otherwise stated all Lie groupoids in this paper will have Hausdorf\/f arrow space.

\subsection{Synthetic dif\/ferential geometry}\label{sec:sdg}

In synthetic dif\/ferential geometry we replace the category ${\rm Man}$ of smooth paracompact Hausdorf\/f manifolds with a certain kind of Grothendieck topos $\mathcal{E}$ called a well-adapted model of synthetic dif\/ferential geometry. In this section we sketch the axioms of a well-adapted model of synthetic dif\/ferential geometry and recall a few key properties.

Firstly there is a full and faithful embedding $\iota\colon {\rm Man} \rightarrowtail \mathcal{E}$ and therefore a ring $R=\iota\mathbb{R}$ in $\mathcal{E}$.
In addition we have the objects
\begin{gather*}D_k=\big\{x\in R\colon x^{k+1}=0\big\},\end{gather*}
which are not terminal. In fact the fundamental Kock--Lawvere axiom holds: the arrow $\alpha\colon R^{k+1}$ $\rightarrow R^{D_k}$ def\/ined by
\begin{gather*}(a_0,a_1,\dots ,a_k)\mapsto\big(d\mapsto a_0+a_1d+\dots +a_kd^k\big)\end{gather*}
is an isomorphism. A set of non-classical objects that will be useful in the sequel are the Weil spectra which are of the following form:
\begin{gather*}{\rm Spec}({\rm Weil}) = \bigg\{(x_1,\dots ,x_n)\colon \bigwedge_{i=1}^n\big(x_i^{k_i}=0\big)\wedge\bigwedge_{j=1}^m(p_j=0)\bigg\},\end{gather*}
where $n,m\in\mathbb{N}_{\geq 0}$, $k_i\in \mathbb{N}_{>0}$ and the $p_j$ are polynomials in the~$x_i$. We write $D_{\infty}=\bigcup_i D_i$ and $D=D_1$.

The following is Def\/inition~3.1 in Part~III of~\cite{MR2244115}.

\begin{Definition}A pair of maps $f_i\colon M_i \rightarrow N$ ($i=1,2$) in ${\rm Man}$ with common codomain are said to be \emph{transversal} to each other if\/f for each pair of points $x_1\in M_1$, $x_2\in M_2$ with $f_1(x_1)=f_2(x_2)$ ($=y$ say), the images of $(df_i)_{x_i}$ ($i=1,2$) jointly span $T_y N$ as a vector space.
\end{Definition}

\begin{Definition}A topos $\mathcal{E}$ together with a full and faithful embedding $\iota\colon Man\rightarrow \mathcal{E}$ is a~\emph{well-adapted model of synthetic differential geometry} if\/f
\begin{itemize}\itemsep=0pt
\item the functor $\iota$ preserves transversal pullbacks,
\item the functor $\iota$ preserves the terminal object,
\item the functor $\iota$ sends arbitrary open covers in ${\rm Man}$ to jointly epimorphic families in~$\mathcal{E}$,
\item the internal ring $\iota(\mathbb{R})$ satisf\/ies the Kock--Lawvere axiom,
\item for all Weil spectra $D_W$ the functor $(-)^{D_W}\colon \mathcal{E}\rightarrow \mathcal{E}$ preserves all colimits.
\end{itemize}
\end{Definition}

\begin{Remark}Since $\iota\colon {\rm Man} \rightarrow \mathcal{E}$ preserves transversal pullbacks it determines an embedding of ${\rm LieGpd}_H$ into ${\rm Grpd}(\mathcal{E})$. Here we have written ${\rm LieGpd}_H$ for the subcategory of ${\rm LieGpd}$ consisting of the groupoids that have Hausdorf\/f arrow space.
\end{Remark}

\begin{Remark}If $M$ is a smooth manifold then we will often abuse notation by writing $M$ to denote the object $\iota(M)$ in the well-adapted model.
\end{Remark}

Using the Kock--Lawvere axiom we can show that $\iota(TM)\cong M^D$ as vector bundles over~$M$ and that the Lie bracket corresponds to an inf\/initesimal commutator. For more detail see~\cite{MR2244115}. Further\-more in Section~\ref{sec:formal-group-laws} we show that formal group laws correspond to groups of the form $(D_{\infty}^n,\mu)$.

\subsection{Smooth af\/f\/ine schemes and the Dubuc topos}\label{sec:dubuc-topos}

In Section~\ref{sec:ordinary-connectedness-implies-enriched} we will need a more detailed description of the coverage that generates the topos~$\mathcal{E}$. Hence in that section we will work in a well-adapted model of synthetic dif\/ferential geometry called the Dubuc topos. In this section we brief\/ly sketch the essential features of the Dubuc topos and refer to~\cite{1981EduardoJDubuc80} for more details. Note that this means that the results of Section~\ref{sec:ordinary-connectedness-implies-enriched} hold for all the well-adapted models generated by a site contained in the Dubuc site. For instance by referring to Appendix~2 of~\cite{MR1083355} we see that our results hold for the Cahiers topos (see~\cite{MR557083}) and the classifying topos of local Archimedean $C^{\infty}$-rings (see Appendix~2 of~\cite{MR1083355}).

In addition in Section~\ref{sec:Lie Groupoids are Integral Complete} it will be convenient to know that every representable object is a~subobject of $R^n$ for some $n\in \mathbb{N}$. Therefore in that section we will work in any well-adapted model $\mathcal{E}$ that is generated by a~subcanonical site whose underlying category is a full subcategory of the category of af\/f\/ine $C^{\infty}$-schemes as def\/ined below. In particular this means that the results of Section~\ref{sec:Lie Groupoids are Integral Complete} hold for the Dubuc topos.
\begin{Definition}\label{def:smooth-affine-schemes}The \emph{category $\mathcal{C}$ of affine $C^{\infty}$-schemes} has as objects pairs $[n,I]$ where $n\in\mathbb{N}$ and~$I$ is a~f\/initely generated ideal of $C^{\infty}(\mathbb{R}^n,\mathbb{R})$. The arrows
\begin{gather*}[n,I]\xrightarrow{f} [m,J]\end{gather*}
are equivalence classes of smooth functions $f\in C^{\infty}(\mathbb{R}^n,\mathbb{R}^m)$ such that
 \begin{itemize}\itemsep=0pt
 \item we identify $f\sim g$ if\/f $f\equiv g \pmod I$,
 \item for all $j\in J$ we have $jf\sim 0$.
 \end{itemize}
\end{Definition}
Now we def\/ine a slight generalisation of the notion of open set. Using these open sets we def\/ine the Dubuc coverage by using inverse images of smooth functions.
\begin{Definition}The \emph{open subobject $U$ of $[n,I]$ defined by $\chi_U\colon \mathbb{R}^n \rightarrow \mathbb{R}$} is the subobject
 \begin{gather*}[n+1,(I,\chi_U\cdot X_{n+1}-1)] \xrightarrow{{\rm proj}} [n,I],\end{gather*}
which intuitively corresponds to the subset $\chi_U^{-1}(\mathbb{R}-\{0\})\cap[n,I]$. The \emph{Dubuc coverage $\mathcal{J}$} consists of the families of open subobjects
 \begin{gather*}(U_i \rightarrowtail [n,I])_{i\in I}\end{gather*}
that are jointly surjective.
\end{Definition}

The site that we use to generate the Dubuc topos is the full subcategory of the category of af\/f\/ine $C^{\infty}$-schemes on the germ-determined schemes which are def\/ined as follows:
\begin{Definition}\label{def:germ-determined-ideal}
For a smooth function $f\colon \mathbb{R}^{n}\rightarrow \mathbb{R}$ we write $\mathbf{g}_{x}(f)$ for the equivalence class of functions that is the germ of $f$ at $x\in\mathbb{R}^{n}$ and $G_{x}$ for the ring of germs of smooth functions at~$x$. For an ideal of smooth functions $I\triangleleft C^{\infty}(\mathbb{R}^{n},\mathbb{R})$ we write $Z(I)$ for the zero-set of $I$ and
\begin{gather*}\mathbf{g}_{x}(I)=\big\{\Sigma_{j=1}^{k}r_{j}\mathbf{g}_{x}(\phi_{j})\colon (r_{j}\in G_{x})\wedge(\phi_{j}\in I)\big\}\end{gather*}
for the ideal generated by germs of elements of $I$. Then a scheme $[n,I]$ is \emph{germ-determined} if\/f
\begin{gather*}\forall\, g\in C^{\infty}\big(\mathbb{R}^n,\mathbb{R}\big), \quad \big(\forall\, x\in Z(I), \, \mathbf{g}_{x}(g)\in \mathbf{g}_{x}(I)\big)\implies g\in I.\end{gather*}
We denote by $\mathcal{C}_{\rm germ}\subset \mathcal{C}$ the full subcategory on the objects that are germ-determined. The \emph{Dubuc topos} is the Grothendieck topos generated by taking sheaves on the site $(\mathcal{C}_{\rm germ},\mathcal{J})$ where~$\mathcal{J}$ is the Dubuc coverage.
\end{Definition}

\subsection{Internal connectedness}\label{sec:enriched-connectedness}

In classical Lie theory we study how much of the data in a Lie groupoid can be recovered from the subset of this data that is inf\/initely close to the identity arrows of the Lie groupoid. Since global features such as connectedness cannot be captured by the inf\/initesimal arrows we need to restrict our attention to Lie groupoids that are source path and source simply connected.

We say that a Lie groupoid $\mathbb{G}$ with arrow space $G$ and object space $M$ is \emph{source path/source simply connected} if\/f all of its source f\/ibres are path/simply connected. Let $\mathbb{I}$ be the pair groupoid on the unit interval $I$ that has precisely one invertible arrow between each pair of elements of~$I$. Then it is easy to see that the global sections of the object $\mathbb{G}^{\mathbb{I}}= {\rm Grpd} (\mathcal{E})(\mathbb{I},\mathbb{G})$ in $\mathcal{E}$ are equivalent to arrows $I \rightarrow G$ in $\mathcal{E}$ that are source constant and start at an identity element of~$G$. Therefore~$\mathbb{G}$ is source path connected if\/f
\begin{gather*}\Gamma\big(\mathbb{G}^{\mathbb{I}}\big) \xrightarrow{\Gamma(\mathbb{G}^{\iota_{\mathbb{I}}})} \Gamma\big(\mathbb{G}^{\partial \mathbb{I}}\big)\end{gather*}
is an epimorphism in ${\rm Set}$. We have written $\Gamma$ for the global sections functor and $\iota_{\mathbb{I}}\colon \partial \mathbb{I} \rightarrow \mathbb{I}$ for the full subcategory that is the pair groupoid on the boundary of~$I$. In this case $\iota_{\mathbb{I}}$ is simply the inclusion of the long arrow $(0,1)\colon \mathbf 2 \rightarrow \mathbb{I}$. Similarly $\mathbb{G}$ is source simply connected if\/f it is source path connected and \begin{gather*}\Gamma\big(\mathbb{G}^{\mathbb{I}^2}\big) \xrightarrow{\Gamma\big(\mathbb{G}^{\iota_{\mathbb{I}^2}}\big)} \Gamma\big(\mathbb{G}^{\partial \mathbb{I}^2}\big)\end{gather*}
is an epimorphism in ${\rm Set}$. We have written $\iota_{\mathbb{I}^2}\colon \partial \mathbb{I}^2 \rightarrow \mathbb{I}^2$ for the full subcategory that is the pair groupoid on the boundary of~$I^2$.

When we work with arbitrary groupoids in a well-adapted model $\mathcal{E}$ of synthetic dif\/ferential geometry it is necessary to work with epimorphisms between objects of $\mathcal{E}$ than between their sets of global sections. Hence we make the following def\/initions:
\begin{Definition} A groupoid $\mathbb{G}$ in $\mathcal{E}$ is \emph{$\mathcal{E}$-path connected} if\/f
 \begin{gather*}\mathbb{G}^{\mathbb{I}} \xrightarrow{\mathbb{G}^{\iota_{\mathbb{I}}}} \mathbb{G}^{\partial \mathbb{I}}\end{gather*}
is an epimorphism in $\mathcal{E}$. A groupoid $\mathbb{G}$ in $\mathcal{E}$ is \emph{$\mathcal{E}$-simply connected} if\/f it is $\mathcal{E}$-path connected and
 \begin{gather*}\mathbb{G}^{\mathbb{I}^2} \xrightarrow{\mathbb{G}^{\iota_{\mathbb{I}^2}}} \mathbb{G}^{\partial \mathbb{I}^2}\end{gather*}
is an epimorphism in $\mathcal{E}$.
\end{Definition}
This means that for an arbitrary groupoid in $\mathcal{E}$ being $\mathcal{E}$-connected is a stronger condition to impose than being source connected. In Section~\ref{sec:path-and-simply-connectedness} we show that a Hausdorf\/f Lie groupoid is source path/simply connected if\/f it is $\mathcal{E}$-path/$\mathcal{E}$-simply connected.

\subsection{The jet part}

The linear approximation of a Lie groupoid has the structure of a Lie algebroid (see for instance Section~3.5 of~\cite{MR2157566}). By contrast in~\cite{burke-relative-to-local-models} we def\/ine an analytic approximation of an arbitrary groupoid in~$\mathcal{E}$. This new structure approximates a Lie groupoid in an analogous way to how a~formal group law approximates a Lie group. In this section we brief\/ly sketch the main features of this analytic approximation.

Using the inf\/initesimal objects of synthetic dif\/ferential geometry we can def\/ine an \emph{infinitesimal neighbour relation~$\sim$}. Intuitively speaking $a\sim b$ expresses that~$b$ is contained in an inf\/initesimal jet based at~$a$. For more details see Section~\ref{sec:infinitesimal-neighbour-relation}. Using this neighbour relation we can def\/ine the \emph{jet part $\mathbb{G}_{\infty}$} of a~groupoid $\mathbb{G}$ with object space $G$ and arrow space $M$ that consists of all the arrows that are inf\/initely close to an identity arrow. In~\cite{burke-relative-to-local-models} we show that this jet part is closed under composition and so def\/ines a subcategory
\begin{gather*}\mathbb{G}_{\infty} \xrightarrow{\iota_{\mathbb{G}}^{\infty}} \mathbb{G},\end{gather*}
which is however not in general a groupoid.

\subsubsection{Symmetry of the neighbour relation}

It turns out that the neighbour relation $\sim$ is not symmetric for all objects of~$\mathcal{E}$. In fact it is not symmetric on the object~$D$ of all nilsquares in the real line. In~\cite{burke-relative-to-local-models} we show that this implies that the jet part~$\nabla{D}_{\infty}$ of the pair groupoid $\nabla{D}$ on $D$ is not a groupoid (although it is a category). Fortunately in~\cite{burke-relative-to-local-models} we also show that the symmetry of $\sim$ in the arrow space of a groupoid $\mathbb{G}$ is not only a necessary condition but also a suf\/f\/icient condition to ensure that the jet part~$\mathbb{G}_{\infty}$ of~$\mathbb{G}$ is a groupoid. We justify this assumption in Section~\ref{sec:symmetry-of-neigbour-relation} by showing that the neighbour relation is symmetric for all classical Hausdorf\/f Lie groupoids.

\subsubsection{Path connectedness of the jet part}

When we prove Lie's second theorem in~\cite{burke-relative-to-local-models} there is an additional condition that is required which is not visible in the classical theory. In addition to requiring that a groupoid satisf\/ies certain connectedness and completeness conditions we need to assume that its jet part is $\mathcal{E}$-path connected. We justify this assumption in Section~\ref{sec:jet-part-path-connected} by showing that the jet part of every classical Hausdorf\/f Lie groupoid is $\mathcal{E}$-path connected.

\subsection{Integral completeness}

To construct global data from local data in classical Lie theory we use the fact that all smooth vector f\/ields admit a unique local solution when we f\/ix an initial vector. Unfortunately when we replace the category ${\rm Man}$ with a well-adapted model $\mathcal{E}$ of synthetic dif\/ferential geometry we can no longer use this result. In this section we identify a class of groupoids for which we can construct global data from local data. It turns out that we do not need to assert the existence of all solutions to smooth vector f\/ields but instead a weaker condition suf\/f\/ices. In~\cite{MR1973056} we see that the crucial lifting property required to prove Lie's second theorem involves the integration of a certain type of path in a Lie algebroid (called $A$-paths) to a certain type of path in a~Lie groupoid (called $G$-paths).

Let $\mathbb{I}$ be the pair groupoid on the unit interval~$I$ and $\mathbb{G}$ be a Lie groupoid with arrow space~$G$ and object space~$M$. In Section~\ref{sec:paths-of-infinitesimals} we show that $A$-paths correspond to global sections of the object $\mathbb{G}^{\mathbb{I}_{\infty}}$ in $\mathcal{E}$ and $G$-paths correspond to global sections of the object~$\mathbb{G}^{\mathbb{I}}$. Hence we restrict attention to groupoids that are integral complete in the following sense:
\begin{Definition}\label{def:integral-complete}
 A groupoid $\mathbb{G}$ in $\mathcal{E}$ is \emph{integral complete} if\/f
 \begin{gather*}\mathbb{G}^{\mathbb{I}} \xrightarrow{\mathbb{G}^{\iota_{\mathbb{I}}^{\infty}}} \mathbb{G}^{\mathbb{I}_{\infty}}\end{gather*}
 is an isomorphism in ${\rm Gpd}(\mathcal{E})$.
\end{Definition}

This assumption is a crucial one in the proof of Lie's second theorem presented in~\cite{burke-relative-to-local-models} and so in Section~\ref{sec:integral-completeness} we justify it by proving that all classical Hausdorf\/f Lie groupoids are integral complete.

\section{Path and simply connectedness}\label{sec:path-and-simply-connectedness}

In this section we show that for all Hausdorf\/f Lie groupoids $\mathbb{G}$ with arrow space $G$ and object space $M$ the classical source path and source simply connectedness conditions coincide with their internal counterparts. (Please see Section~\ref{sec:enriched-connectedness} for the relevant def\/initions.) In other words, we show that if $\mathbb{G}$ is source path connected then~$\mathbb{G}$ is~$\mathcal{E}$-path connected and if further~$\mathbb{G}$ is source simply connected then~$\mathbb{G}$ is $\mathcal{E}$-simply connected. To do this we will need an explicit description of the coverage that generates the well-adapted model~$\mathcal{E}$. Hence for this section we will assume that~$\mathcal{E}$ is the Dubuc topos which is generated by the Dubuc site as def\/ined in Section~\ref{sec:dubuc-topos}. Note that this means that our results hold for all the well-adapted models generated by a site contained in the Dubuc site. For instance by referring to Appendix~2 of~\cite{MR1083355} we see that our results hold for the Cahiers topos (see \cite{MR557083}) and the classifying topos of local Archimedean $C^{\infty}$-rings (see Appendix~2 of~\cite{MR1083355}).

We deduce both the path connected and simply connected results from the following stronger result.

\begin{Notation} Let $B$ be a compact and contractible subset of a Euclidean space that is a zero set of an ideal of smooth functions~$I$:
\begin{gather*}B=[n,I]=\big\{\vec{x}\in \mathbb{R}^n\colon \forall\, \phi \in I,\, \phi(\vec{x})=0\big\},\end{gather*}
which means that we can view $B$ as a representable object in the Dubuc topos as well as a~subset of Euclidean space.

Let $\partial B$ denote the boundary of $B$ and $\nabla B$ and $\nabla\partial B$ be the pair groupoids on $B$ and $\partial B$ respectively. (Recall that the pair groupoid has precisely one invertible arrow between any pair of objects.) There is a~natural inclusion $\iota_B\colon \nabla\partial B \rightarrow \nabla B$.
\end{Notation}

\begin{Notation} We write $r\in_{X}R$ to denote that $r$ is an arrow $X\rightarrow R$ in~$\mathcal{E}$ and say that~$r$ is \emph{a~generalised element of $R$ at stage of definition~$X$}.
\end{Notation}

We prove that if every global element $f\in_1 \mathbb{G}^{\nabla\partial B}$ has a f\/iller $F\in_1 \mathbb{G}^{\nabla B}$ (i.e., $G^{\iota_B}F=f$) then the arrow
\begin{gather*}\mathbb{G}^{\iota_B}\colon \ \mathbb{G}^{\nabla B} \rightarrow \mathbb{G}^{\nabla\partial B}\end{gather*}
is an epimorphism in $\mathcal{E}$. Note that being $\mathcal{E}$-path connected is the case when~$B$ is the unit interval~$I$ and being $\mathcal{E}$-simply connected is the conjunction of the cases $B=I$ and $B=I^2$.

Our general strategy will be to split the tangent bundle using the submersion~$s$ and then show that various constructions involving Riemannian exponential maps can be forced to respect this splitting. Once this is done we can work in just one source f\/ibre where the result is substantially easier.

However f\/irst we need to consider the interrelationships between various kinds of open subset and subobject possible in the context of a well-adapted model of synthetic dif\/ferential geometry.

\subsection{Open subobjects of function spaces}

Our aim is to show that a certain arrow between function spaces is an epimorphism. As is the case for all Grothendieck toposes, the epimorphisms in the Dubuc topos are characterised in terms of the coverage that is used to generate the topos. However the natural and convenient notion of open subset for a space of smooth functions is the smooth compact-open topology. In order to mediate between these two notions of open subobject/subset we introduce another type of open subobject due to Penon. First we recall that every Penon open subobject of a~representable object is a~Dubuc open subobject. Then we recall that any smooth compact-open subset of the set of global sections of a~function space induces a Penon open subobject of the function space.

\subsubsection{Penon open subobjects}

In this section we will brief\/ly sketch some of the theory of topological structures in synthetic dif\/ferential geometry and refer to~\cite{1987BungeMartaandDubucEduardoJ131} and~\cite{MR798526} for more comprehensive accounts. Following Penon in~\cite{MR609161} we say that an element $r$ of the line object $R$ in the Dubuc topos is \emph{infinitesimal} if\/f
\begin{gather*}\neg\neg(r=0)\end{gather*}
holds in the internal logic of the Dubuc topos. Since the line object contains nilpotent elements it is not a f\/ield. However it is easy to see that all of these nilpotents are inf\/initesimal as def\/ined above and in fact Theorem~10.1 in~\cite{MR2244115} tells us that the line object~$R$ is a~\emph{field of fractions}, which is to say that for all elements $r\in R$ the proposition
\begin{gather}\label{field-of-fractions}
 \neg(r=0)\iff (r~\text{is invertible})
\end{gather}
holds. We note that in the context of classical logic being a f\/ield of fractions implies that every element is either zero or invertible but this implication does not hold for intuitionistic logic. Using the correspondence in \eqref{field-of-fractions} we can deduce that the inf\/initesimals and the invertible elements of the line object are separated in the following sense: for all $r,s\in R$ the proposition
\begin{gather*}\neg(r=0)\wedge\neg\neg(s=0)\implies \neg(r=s)\end{gather*}
holds. The following def\/inition is Def\/inition~1.5 in~\cite{MR866615}.
\begin{Definition}\label{04/16/1505:00:39 PM}
A subobject $U\subset X$ in $\mathcal{E}$ is \emph{Penon open} if\/f the proposition
\begin{gather*}\forall\, u\in U, \quad \forall\, x\in X, \quad (x\in U)\vee(\neg(u=x))\end{gather*}
holds in the internal logic of $\mathcal{E}$.
\end{Definition}
\begin{Remark}If $U$ is Penon open then for all $u\in U$ then there is an inclusion
\begin{gather*}\{x\colon \neg\neg (x=u)\}\subset U.\end{gather*}
\end{Remark}
\begin{Example}
The subobject
\begin{gather*}\{r\in R\colon \neg(r=0)\}\end{gather*}
of $R$ is a Penon open subset.
\end{Example}In fact we can give a characterisation of the Penon open subsets of representable objects in terms of classical open subsets. Recall from \cite{1981EduardoJDubuc80} and Lemma~1.3 in~III.1 of~\cite{MR1083355} that the site generating the Dubuc topos is subcanonical and that by construction the image of any open subset under the full and faithful embedding of the category of smooth manifolds into the Dubuc topos is a Dubuc open subobject. The following is Corollary~8 in~\cite{MR1055025}.
\begin{Proposition}\label{Penon-iff-Dubuc-on-representables}
A subobject $U$ of a representable object $yX$ in the Dubuc topos is Penon open iff it is of the form $U=\iota V\cap yX$ for some open subset $V\subset\mathbb{R}^{n}$.
\end{Proposition}
\begin{Corollary}\label{cor:penon-implies-dubuc}
 If $X$ is a Penon open subobject of a representable then it is Dubuc open.
\end{Corollary}
Finally we record the result that arbitrary Penon open subobjects are stable under pullback.

\begin{Corollary}
 Let $f \colon W \rightarrow X$ be an arrow and $ U $ be a Penon open subobject of $ X $. Then
\begin{gather*}
 f ^* U = \{ w \in W \colon f w \in U \}
\end{gather*}
is a Penon open subobject of $W$.
\end{Corollary}
\begin{proof}
 The hypothesis that $ U $ is a Penon open subobject of $ X $ implies that for all $w \in W$ and $v \in f ^* U$ the proposition
 \begin{gather*}
 ( fw \in U ) \vee ( \neg ( fw = fv ) )
 \end{gather*}
 holds. But by def\/inition $ fw \in U $ if\/f $ w \in f^*U $ and it is immediate that $ \neg ( fw = fv )$ implies that $ \neg (w = v )$.
\end{proof}

\subsubsection{Smooth compact-open subsets}\label{sec:smooth-compact-open-subsets}

Recall that for any topos $ \mathcal{E}$ the global sections functor $\Gamma$ restricts to a functor from the poset of subobjects of some function space $Y^X$ to the poset of subsets of $\Gamma(Y^X)$. In fact when $\mathcal{E}$ is the Dubuc topos $\Gamma$ has a right adjoint
\begin{equation*}
\includegraphics[scale=0.92]{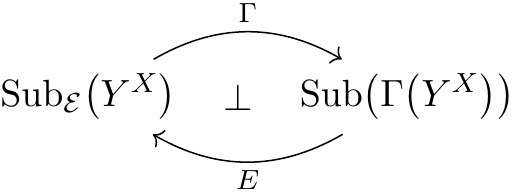}
\end{equation*}
for which the unit is an isomorphism (see Lemma~1.5 in~\cite{MR796353}). Therefore it is natural to ask whether we can characterise the subsets~$U$ of $\Gamma(Y^X)$ for which $EU$ is a~Penon open subobject. Let~$M$,~$N$ be smooth manifolds and $ f \colon M \rightarrow N $ a smooth function. Let $ ( \phi , V ) $ and $ ( \psi , W ) $ be charts for $ M $ and $ N $ respectively and let $ K $ be a compact subset of $ V $ such that $ f ( K ) \subset W $. Then following Section~2.1 in~\cite{MR0448362} we def\/ine the \emph{weak subbasic neighbourhood} $ U_{f , \epsilon, K } ^{ \phi , \psi} $ of $ C^{ \infty } ( M , N ) $ to be the set of all smooth functions $ g \colon M \rightarrow N $ such that $ g ( K ) \subset V $ and such that for all integers~$ k $ the inequality
\begin{gather*}\big|D^k\big(\psi f \phi^{-1}\big) (m) - D^k \big( \psi g \phi^{-1} \big) (m)\big| < \epsilon\end{gather*}
holds for all $m\in\phi(K)$. We call the topology generated by the weak subbasic neighbourhoods the \emph{smooth compact-open topology}. The following is Proposition~1.6 of~\cite{MR796353}.
\begin{Proposition}\label{prop:compact-open-implies-penon}
If $U$ is a smooth compact-open subset of $\Gamma(Y^X)$ then $EU$ is a Penon open subobject.
\end{Proposition}
\begin{Remark}If $f\in U$ then $f\colon 1 \rightarrow Y^X$ factors through $EU \rightarrowtail Y^X$.
\end{Remark}

\subsection{Splitting the tangent bundle}

In this section we use the source submersion $s\colon G\rightarrow M$ of a Lie groupoid $G \rightrightarrows M $ to split the tangent bundle $TG\rightarrow G$ into horizontal and vertical components. We then conf\/irm that when we pullback the tangent bundle to a contractible base space this splitting is maintained in the trivial bundle that results.

\begin{Definition}If $p_H\colon H\rightarrow G$ and $p_V\colon V\rightarrow G$ are vector bundles then the \emph{direct sum} $H\oplus V$ is def\/ined as the pullback vector bundle
\begin{equation*}
\begin{tikzcd}
H\oplus V \dar{p_G} \rar{\Delta_G} & H\times V \dar{p_H\times p_V}\\
G \rar{\Delta} & G\times G,
\end{tikzcd}
\end{equation*}
where $\Delta$ is the diagonal.
\end{Definition}

\begin{Notation}Let $\mathbb{G}$ be a Lie groupoid with arrows space $G$, object space $M$, source map~$s$ and a Riemannian metric $\sigma^G$ on~$G$. Let~$B$ be a~contractible subset of a Euclidean space $\mathbb{R}^n$ that is the zero set of an ideal of smooth functions. Let $p_V\colon V=\ker(s^D)\rightarrow G$ and $p_H\colon H=(\ker(s^D))^{\perp}$ $\rightarrow G$ be vector bundles where the orthogonal complement is def\/ined using the metric~$\sigma$.
\end{Notation}

First we note that the squares
\begin{equation*}
\begin{tikzcd}
B\times R^m \rar{\tau_G} \dar{\pi_0} & H\oplus V \dar{p_G}\\
B \rar{F} & G,
\end{tikzcd}
\qquad
\begin{tikzcd}
B\times R^k \rar{\tau_H} \dar{\pi_0} & H \dar{p_H}\\
B \rar{F} & G
\end{tikzcd}
\qquad \text{and}\qquad
\begin{tikzcd}
B\times R^n \rar{\tau_V} \dar{\pi_0} & V \dar{p_V}\\
B \rar{F} & G
\end{tikzcd}
\end{equation*}
are pullbacks for some natural numbers $m$, $k$ and $n$ because $B$ is contractible. By construction $TG\cong H\oplus V$ as vector bundles.

\begin{Lemma}\label{lem:split-s}
The square
\begin{equation*}\begin{tikzcd}
B\times R^k \times R^n \rar{\tau_G} \dar{\pi_{0,1}} & H\oplus V \dar{\pi_0\circ\Delta_G}\\
B\times R^k \rar{\tau_H} & H
\end{tikzcd}\end{equation*}
commutes.
\end{Lemma}
\begin{proof}There is a unique $\psi\colon B\times R^k\times R^n\rightarrow H\oplus V$ making
\begin{equation*}
\begin{tikzcd}
B\times R^k\times R^n \arrow{ddd}{\pi_0} \drar[dashed]{\psi} \arrow{rrr}{(b,v,b,u)} &&& (B\times R^k)\times(B\times R^n) \arrow{ddd}{\pi_0\times\pi_0} \dlar{\tau_H\times\tau_V}\\
& H\oplus V \dar{p} \rar{\Delta_G} & H\times V \dar{p_H\times p_V} &\\
& G \rar{\Delta} & G\times G & \\
B \urar{F} \arrow{rrr}{\Delta} &&& B\times B \ular{F\times F}
\end{tikzcd}
\end{equation*}
into a commutative cube because the centre square is a pullback. Furthermore the right and outer squares are easily seen to be pullbacks. This means that the left square is a pullback and $\psi=\tau_G$. Now the result follows from the fact that
\begin{equation*}\begin{tikzcd}
B\times R^k \times R^n \rar{(b,v,b,u)} \dar{\tau_G} & (B\times R^k)\times(B\times R^n) \dar{\tau_H\times \tau_V} \rar{\pi_0} &B\times R^k \dar{\tau_H} &\\
H\oplus V \rar{\Delta_G} & H\times V \rar{\pi_0} & H
\end{tikzcd}\end{equation*}
commutes.
\end{proof}

\begin{Corollary}\label{cor:split-tangent}
There is an isomorphism $\alpha$ making
\begin{equation*}\begin{tikzcd}
B\times R^k\times R^n \rar{\tau_G} \dar{\pi_{0,1}}& G^D \dar{s^D}\\
B\times R^k \rar{\alpha\circ\tau_H} & M^D.
\end{tikzcd}\end{equation*}
\end{Corollary}
\begin{proof}
Since $s^D$ is an epimorphism there is an isomorphism $\alpha$ making
\begin{equation*}\begin{tikzcd}
H\oplus V \arrow{rr}{s^D} \drar[swap]{\pi_0\circ \Delta_G} &&M^D \\
&H \urar{\alpha}&
\end{tikzcd}\end{equation*}
commute. Then result follows immediately from Lemma~\ref{lem:split-s}.
\end{proof}

\subsection{Riemannian submersions}

In order to transfer f\/illers between neighbouring source f\/ibres we need to know how to transport them in parallel to the source f\/ibres. To do this we will use the exponential map on the arrow space $G$ induced by a Riemannian metric on~$G$. However it is not in general true that for arbitrary Riemannian metrics~$\eta^G$ and~$\eta^M$ on~$G$ and~$M$ respectively that~$s$ maps geodesics with respect to~$\eta^G$ to geodesics with respect to~$\eta^M$.

We now recall a little of the theory of Riemannian submersions which will allow us to construct Riemannian metrics $\sigma^G$ on $G$ and $\sigma^M$ on $M$ such that $s$ maps geodesics with respect to $\sigma^G$ to geodesics with respect to $\sigma^M$.

\begin{Notation} Let $G$ and $M$ be smooth manifolds with Riemannian metrics $\eta^G$ and $\eta^M$ respectively. Let $s\colon G\rightarrow M$ be a submersion. We write $\ker(s)=V$ for the sub-bundle of the tangent bundle $G^D$ that is parallel to the $s$-f\/ibres and $H=(\ker(s))^{\perp}$ for the bundle orthogonal to $V$ with respect to the Riemannian metric $\eta^G$. Let $U^G$ and $U^M$ denote the domains of the exponential maps associated to $\eta^G$ and $\eta^M$ respectively.
\end{Notation}

The next def\/inition is part of the Def\/inition 26.9 in \cite{Michor2008}.

\begin{Definition} The submersion $s$ is a \emph{Riemannian submersion} if\/f
\begin{gather*}\big(s^D|_H\big)_g\colon \ H_g \rightarrow \big(M^D\big)_{s(g)}\end{gather*}
is an isometric isomorphism.
\end{Definition}

The next result is Lemma 2.1.1 in~\cite{delHoyo2015}.

\begin{Lemma} If $s\colon G \rightarrow M$ is a submersion then we can choose Riemannian metrics~$\sigma^G$ and~$\sigma^M$ on~$G$ and~$M$ respectively that make $s$ a Riemannian submersion.
\end{Lemma}
\begin{proof} To begin with choose arbitrary Riemannian metrics $\eta^G$ and $\eta^M$ on $G$ and $M$ respectively. Use $\eta^G$ to decompose $G^D = (\ker(s))^{\perp}\oplus \ker(s)=H\oplus V$. Now we can def\/ine an alternative positive def\/inite inner product $\sigma^H$ on $H$ as the pullback of $\eta^M$ along the isometry $(s^D|_H)$. Also we can restrict the Riemannian metric~$\eta^G$ to a~positive def\/inite inner products $\sigma^V$ on~$V$. Then we def\/ine a new Riemannian metric~$\sigma^G$ on~$G$ by declaring all vectors in~$H$ to be orthogonal to all vectors in~$V$. By construction~$s$ is a Riemannian submersion with respect to~$\sigma^G$ and~$\eta^M$.
\end{proof}

\begin{Lemma}\label{lem:geodesics-preserved}
If $s\colon G \rightarrow M$ is a Riemannian submersion and $c\colon [0,1]\rightarrow G$ is a geodesic in~$G$ such that $c'(0)\in H$ then $s\circ c$ is a geodesic in~$M$. Furthermore if $a\in [0,1]$ then $c'(a)\in H$.
\end{Lemma}
\begin{proof}
This is Corollary~26.12 in~\cite{Michor2008}.
\end{proof}

\begin{Corollary}\label{exp-and-source}
If $s\colon G \rightarrow M$ is a Riemannian submersion then there exist open subsets $W^M\subset U^M$ and $W^G\subset U^G$ such that
\begin{equation*}
\begin{tikzcd}
W^G \rar{\exp^G} \dar{s^D|_{W^G}} & G \dar{s}\\
W^M \rar{\exp^M} & M
\end{tikzcd}
\end{equation*}
commutes.
\end{Corollary}
\begin{proof}
Immediate from Lemma~\ref{lem:geodesics-preserved}. See also Proposition~5.9 in~\cite{delHoyo2015}.
\end{proof}

\subsection{Constructing a tubular extension}\label{tubular-extension}

In this section we construct a tubular extension $B\times C^k\times C^n\rightarrow G$ for every smooth map $F\colon B\rightarrow G$ where $G$ is the arrow space of a Hausdorf\/f Lie groupoid. In the next two sections we work within this tubular extension to construct the Penon open subobject that we need. We also show that this extension commutes in the appropriate way with the source map.

\begin{Notation}Let $C^n$ denote the open unit hypercube in $R^n$. Let $B$ be a contractible and compact subset of Euclidean space that is the zero set of an ideal of smooth functions.
\end{Notation}

\begin{Lemma}\label{open-inclusions}
If $\overline{W}^G=W^G \cap (s^D)^{-1}(W^M)$ then there are open inclusions $\nu^G\colon C^k\rightarrow R^k$ and $\nu^M\colon C^n\rightarrow R^n$ such that $\nu^G(\vec{0})=\vec{0}$ and $\nu^M(\vec{0})=\vec{0}$ and maps $\iota^G\colon B\times C^k\times C^n\rightarrow \overline{W}^G$ and $\iota^M\colon B\times C^k\rightarrow W^M$ such that
\begin{equation*}\begin{tikzcd}
B\times C^k\times C^n \rar{\iota^G} \dar{B\times \nu^M\times \nu^G} & \overline{W}^G \dar\\
B\times R^k\times R^n \rar{\tau_G} & G^D
\end{tikzcd}\qquad \text{and}\qquad
\begin{tikzcd}
B\times C^k \rar{\iota^M} \dar{B\times \nu^M} & W^M \dar \\
B\times R^k \rar{\alpha\circ\tau_H} & M^D
\end{tikzcd}\end{equation*}
commute.
\end{Lemma}
\begin{proof} By construction for each $b\in B$ the arrow $\tau_G(b,-,-)\colon R^k\times R^n\rightarrow G^D _b$ is an isomorphism. Therefore $X_b = \tau_G(b,-,-)^{-1}(\overline{W}^G _b)$ specif\/ies a collection of open sets containing $\vec{0}$ in $R^{k+n}$ that vary smoothly with $B$. Since $B$ is compact we can f\/ind an open ball around $\vec{0}$ contained in each of the $X_b$. Now the existence of $\iota^G$ and $\nu^G$ follows easily. The existence of $\iota^M$ and $\nu^M$ follows similarly.
\end{proof}

\begin{Lemma}If $F\colon B\rightarrow G$ is a smooth map that is $s$-constant and starts at an identity arrow then there exist smooth maps $\xi^G\colon B\times C^k\times C^n\rightarrow G$ and $\xi^M\colon B\times C^k\rightarrow M$ such that
\begin{itemize}\itemsep=0pt
\item both $\xi^G(b,0,0)=F(b)$ and $\xi^N(b,0,0)=F(b)$,
\item for all $b\in B$ both $\xi^G(b,-,-)$ and $\xi^M(b,-,-)$ are open inclusions,
\item the diagram
\begin{equation*}\begin{tikzcd}
B\times C^k\times C^n \dar{\pi_{0,1}} \rar{\xi^G} & G \dar{s}\\
B\times C^k \rar{\xi^M} & M
\end{tikzcd}\end{equation*}
commutes.
\end{itemize}
\end{Lemma}
\begin{proof} If $\overline{W}^G=W^G \cap (s^D)^{-1}(W^M)$ then in the cube
\begin{equation*}\begin{tikzcd}
B\times C^n\times C^k \drar \arrow{rrr}{\iota^G} \arrow{ddd}{\pi_{0,1}} &&& \overline{W}^G \dlar \arrow{ddd}{s^D|_{\overline{W}^G}}\\
&B\times R^k\times R^n \dar{\pi_{0,1}} \rar{\tau_G} & G^D \dar{s^D} &\\
& B\times R^k \rar{\alpha\circ\tau_H} & M^D &\\
B\times C^k \urar \arrow{rrr}{\iota^M} &&& W^M \ular[rightarrowtail]
\end{tikzcd}\end{equation*}
the centre square commutes by Corollary~\ref{cor:split-tangent}, the upper and lower squares commute by Lemma~\ref{open-inclusions} and the left and right squares commute by construction. Therefore the outer square commutes because~$W^M\rightarrowtail M^D$ is a monomorphism. The result now follows from \mbox{pasting} the square shown to commute in Corollary~\ref{exp-and-source} onto the right of the above square; the maps we require are $\xi^G = \exp^G\circ \iota^G$ and $\xi^M = \exp^M\circ\iota^M$.
\end{proof}

\subsection{A subobject of the tubular extension admitting f\/illers}\label{subobject-with-fillers}

In the previous section we constructed a tubular extension $\xi^G _F\colon B\times C^k\times C^n\rightarrow G$ for every $F\in \mathbb{G}^{\nabla B}$. In this section we construct a subobject of $\mathbb{G}^{\nabla \partial B}$ from this tubular extension such that every element of this subobject admits a f\/iller. In the next section we f\/ind a Penon open subobject contained in this subobject.

\begin{Notation} Let $B$ be a subset of Euclidean space that is the zero set of an ideal of smooth functions. Let $f\in \mathbb{G}^{\nabla\partial B}$ have a f\/iller $F\in \mathbb{G}^{\nabla B}$. We write $\xi^G_F$ for the tubular extension constructed in Section~\ref{tubular-extension}.
\end{Notation}

\begin{Remark} For all $\vec{x}_0\in C^k$ the map $\partial B\rightarrow G$ def\/ined by $b\mapsto \xi^G_F(b,\vec{x}_0,\vec{0})$ has f\/iller $B\rightarrow G$ def\/ined by $b\mapsto \xi^G_F(b,\vec{x}_0,\vec{0})$.
\end{Remark}

\begin{Definition} The subobject $T_f \rightarrowtail \mathbb{G}^{\nabla\partial B}$ consists of all $\chi\in \mathbb{G}^{\nabla\partial B}$ such that
\begin{gather*}\forall\, b\in B, \quad \chi(b)\in \xi^G_F\big(b,C^k,C^n\big)\end{gather*}
or equivalently $T_f$ is the subobject of $\mathbb{G}^{\nabla\partial B}$ such that
\begin{gather*}\forall\, b \in B, \quad \exists\, \vec{x}_0\in C^k, \quad \exists\, h_{\chi}\in \big(C^n\big)^{\partial B}, \quad \chi(b)=\xi^G_F\big(b,\vec{x}_0,h_{\chi}(b)\big),\end{gather*}
because $\chi$ is source constant.
\end{Definition}

\begin{Remark}By construction $F(b)=\xi^G_F(b,\vec{0},\vec{0})$. Restricting to $\partial B$ gives that $f(b)=\xi^G_F(b,\vec{0},\vec{0})$ and hence $f\in T_f$.
\end{Remark}

\begin{Lemma}If $\chi\in T_f$ then $\chi$ has a filler $X\in \mathbb{G}^{\nabla B}$.
\end{Lemma}
\begin{proof}If $\chi(b)=\xi^G_F(b,\vec{x}_0,h_{\chi}(b))$ then there is an homotopy from $\chi$ to $(b\mapsto \xi^G_F(b,\vec{x}_0,\vec{0}))$ def\/ined by
\begin{align*}
I\times \partial B &\rightarrow G,\\
(a,b) &\mapsto \xi^G_f\left(b,\vec{x}_0,(1-a)h_{\chi}(b)\right),
\end{align*}
and composing this homotopy with the f\/iller $(b\mapsto \xi^G_F(b,\vec{x}_0,\vec{0}))$ is a f\/iller for $\chi$.
\end{proof}

\subsection{A compact-open set inside a tubular extension}\label{sec:compact-open-in-tubular}

In this section we identify a compact-open set that is contained in space of global sections of~$\mathbb{G}^{\nabla\partial B}$ that is contained in the subobject $T_f$ constructed in Section~\ref{subobject-with-fillers}. Once we have done this we can deduce using Proposition~\ref{prop:compact-open-implies-penon} the existence of a Penon open subobject $V_f$ of $\mathbb{G}^{\nabla\partial B}$ such that all maps in $V_f$ have f\/illers.

\begin{Notation} Let $B$ be a subset of Euclidean space that is the zero set of an ideal of smooth functions. Let $f\in \mathbb{G}^{\nabla\partial B}$ have a f\/iller $F\in \mathbb{G}^{\nabla B}$. We write $\xi^G_F$ for the tubular extension constructed in Section~\ref{tubular-extension}. Let $D^n \rightarrowtail C^n$ be the inclusion of the ball of radius $\frac{1}{2}$ centred at the origin. Let $E^n \rightarrowtail D^n$ be the inclusion of the ball of radius $\frac{1}{4}$ centred at the origin,
\end{Notation}

\begin{Definition}The compact-open subset $W_f$ of $\Gamma(\mathbb{G}^{\nabla\partial B})$ is def\/ined as follows. Let $U_b=f^{-1}\xi^G_F(b,E^k,E^n)$. Now $(U_b)_{b\in \partial B}$ covers $\partial B$ because $b\in U_b$. Since $\partial B$ is compact we can choose $b_1,\dots ,b_n\in \partial B$ such that $(U_{b_i})_{i=1}^n$ covers $\partial B$. The compact-open set $W_f$ that we require is def\/ined by the family $(\overline{U_{b_i}},\xi^G_F(b_i,D^k,D^n))_{i=1}^n$.
\end{Definition}

\begin{Remark}Note that $f\in W_f$ because $f(U_b) = \xi^G_F(b,E^k,E^n)$ and so $f(\overline{U_b}) \subset \xi^G_F(b,D^k,D^n)$.
\end{Remark}

\begin{Lemma} The compact-open set $W_f$ of $\Gamma(\mathbb{G}^{\nabla\partial B})$ is contained in $\Gamma(T_f)$.
\end{Lemma}
 \begin{proof}Let $\chi\in W_f$. For each $b\in \partial B$ there exists at least one $i\in\{1,\dots ,n\}$ such that $b\in U_{b_i}$. For all such $i$ the elements $f(b)$ and $\chi(b)$ are in the open set $\xi^G_F(b_i,D^k,D^n)$ of $G$. Now for each $b\in\partial B$ the map $\iota^G(b,-,-)$ preserves distances. Furthermore since $\xi^G_F=\exp^G\circ \iota^G$ the map $\xi^G_F(b,-,-)$ preserves distances from the origin. Finally recall that $f(b_i)=\xi^G_F(b_i,\vec{0},\vec{0})$. Hence
 \begin{gather*}d(\chi(b),f(b))\leq d(\chi(b),f(b_i))+d(f(b_i),f(b))<\tfrac{1}{2}+\tfrac{1}{2}=1,\end{gather*}
 which tells us that in fact $\chi\in \Gamma(T_f)$.
 So $W_f\subset \Gamma(T_f)$.
 \end{proof}

\begin{Corollary}\label{main-lemma}
If $f\in \mathbb{G}^{\nabla\partial B}$ has filler $F\in \mathbb{G}^{\nabla B}$ then there exists a Penon open subobject $\Phi\colon V_f\rightarrowtail \mathbb{G}^{\nabla\partial B}$ and a lift $\Psi\colon V_f\rightarrow \mathbb{G}^{\nabla B}$ making
\begin{equation*}\begin{tikzcd}
& \mathbb{G}^{\nabla B} \dar{\mathbb{G}^{\iota_B}} \\
V_f \rar[swap]{\Phi_f} \urar[dashed]{\Psi_f} & \mathbb{G}^{\nabla\partial B}
\end{tikzcd}\end{equation*}
 commute.
\end{Corollary}
\begin{proof}Let $V_f=E(W_f)$ where $E$ is the left adjoint to the global sections functor as in Section~\ref{sec:smooth-compact-open-subsets}. Note that by construction $f\in \Gamma(V_f)$.
\end{proof}

\subsection{Ordinary connectedness implies internal connectedness}\label{sec:ordinary-connectedness-implies-enriched}

Now we are in a position to deduce the main result of this paper. Let $\mathbb{G}$ be a~(Hausdorf\/f) Lie groupoid with arrow space~$G$ and object space~$M$.

\begin{Theorem} If $B$ is a compact and contractible subset of Euclidean space that is the zero set of an ideal of smooth functions then the arrow $\mathbb{G}^{\iota_B}\colon \mathbb{G}^{\nabla B} \rightarrow \mathbb{G}^{\nabla\partial B}$ is an epimorphism.
\end{Theorem}

\begin{proof}We perform a sequence of reductions to show that it in fact suf\/f\/ices to prove Corollary~\ref{main-lemma}.

Firstly, to show that $\mathbb{G}^{\iota_B}$ is an epimorphism, it will suf\/f\/ice to show that for all representable objects $X$ in $\mathcal{E}$ and arrows $\phi\colon X \rightarrow \mathbb{G}^{\nabla\partial B}$ in $\mathcal{E}$ there exists a Dubuc open cover $(\iota_i\colon X_i \rightarrow X)_{i\in I}$ such that for all $i\in I$ there exists a lift $\psi_i$ making
\begin{equation*}
\begin{tikzcd}
 {} & \mathbb{G}^{\nabla B} \dar{\mathbb{G}^{\iota_B}}\\
 X_i \rar[swap]{\phi\iota_i} \urar[dashed]{\psi_i} & \mathbb{G}^{\nabla\partial B}
\end{tikzcd}
\end{equation*}
commute.

In fact it will suf\/f\/ice to f\/ind for each $f\in_1 \mathbb{G}^{\nabla\partial B}$ a Penon open subobject $U_f$ of $\mathbb{G}^{\nabla\partial B}$ containing $f$ and a lift
\begin{equation}\label{main-theorem}
\begin{tikzcd}
 {} & \mathbb{G}^{\nabla B} \dar{\mathbb{G}^{\iota_B}}\\
 U_f \rar[swap]{\phi_f} \urar[dashed]{\psi_f} & \mathbb{G}^{\nabla\partial B}.
\end{tikzcd}
\end{equation}
Indeed $(U_f)_{f\in \mathbb{G}^{\nabla\partial B}}$ covers $\mathbb{G}^{\nabla\partial B}$ as Penon open subobjects and so the pullback cover \linebreak $(\phi^{-1}(U_f))_{f\in\mathbb{G}^{\nabla\partial B}}$ covers $X$ as Penon open subobject. But now we use Corollary~\ref{cor:penon-implies-dubuc} and the fact that $X$ is representable to see that $(U_f)_{f\in \mathbb{G}^{\nabla\partial B}}$ covers $\mathbb{G}^{\nabla\partial B}$ as~Dubuc open subobjects also.

But the existence of $\psi_f$ and a Penon open $\phi_f$ making \eqref{main-theorem} commute is the conclusion of Corollary~\ref{main-lemma}.
\end{proof}

\begin{Corollary}If $\mathbb{G}$ is an $s$-path connected Lie groupoid then the arrow $\mathbb{G}^{\iota_I}\colon \mathbb{G}^{\nabla I}\rightarrow \mathbb{G}^{\mathbf{2}}$ is an epimorphism and so, by definition, the groupoid $\mathbb{G}$ is internally path connected.
\end{Corollary}

\begin{Corollary}If $\mathbb{G}$ is an $s$-simply connected Lie groupoid then the arrow $\mathbb{G}^{\iota_{(I\times I)}}\colon \mathbb{G}^{\nabla (I\times I)}\rightarrow \mathbb{G}^{\nabla\partial(I\times I)}$ is an epimorphism and so, by definition, the groupoid $\mathbb{G}$ is internally simply connected.
\end{Corollary}

\section{Properties of the jet part}

\subsection{The inf\/initesimal neighbour relation}\label{sec:infinitesimal-neighbour-relation}

In this section we introduce the inf\/initesimal neighbour relation which is used to def\/ine the jet part of a category in~\cite{burke-relative-to-local-models}. If $\mathbb{C}$ is a category in any well-adapted model $\mathcal{E}$ of synthetic dif\/ferential geometry and $M$ is the space of objects of $\mathbb{C}$ then we def\/ine the inf\/initesimal neighbour relation on objects of the slice topos~$\mathcal{E}/M$. In~\cite{burke-relative-to-local-models} we justify this choice by showing that the jet part def\/ined using this neighbour relation is closed under composition in~$\mathbb{C}$.

Let $a,b\colon X \rightarrow B$ where $X$ and $B$ are objects of the topos~$\mathcal{E}/M$. Then $a\sim b$ if\/f there exists a cover $(\iota_{i}\colon X_{i}\rightarrow X)_{i\in I}$ in ${\mathcal E}/M$ such that for each $i$ there exists an object $D_{W_{i}}\in {\rm Spec}({\rm Weil})$, an arrow $\phi_{i}\colon X_{i}\times D_{W_{i}}\rightarrow B$ and an arrow $d_{i}\colon X_{i}\rightarrow D_{W_{i}}$ such that
\begin{equation*}\begin{tikzcd}
X_{i} \dar{a_{i}} \rar{(1_{X_{i}},0)} & X_{i}\times D_{W_{i}} \dar{\phi_{i}} \\
B \rar{1_{B}} & B
\end{tikzcd}\end{equation*}
and
\begin{equation*}\begin{tikzcd}
X_{i} \dar{b_{i}} \rar{(1_{X_{i}},d_{i})} & X_{i}\times D_{W_{i}} \dar{\phi_{i}} \\
B \rar{1_{B}} & B
\end{tikzcd}\end{equation*}
commute, where we have written $a_i$ and $b_i$ for the restrictions of $a$ and $b$ to $X_i$.
\begin{Remark}
The relation $\sim$ is not always symmetric. In fact in \cite{burke-relative-to-local-models} we see that $\sim$ is not symmetric in the case $B=D$ and $M=1$.
\end{Remark}

The relation $\approx$ is the transitive closure of $\sim$ in the internal logic of $\mathcal{E}/M$. This means that for $a,b\colon X \rightarrow B$ we have $a\approx b$ if\/f there exists a cover $(\iota_{i}\colon X_{i}\rightarrow X)_{i\in I}$ and for each $i$ there exists a natural number $n_{i}$ and elements $x_{i_{0}},x_{i_{1}},\dots ,x_{i_{n_{i}}}\in_{X_{i}}B$ such that
\begin{gather*}a_{i}=x_{i_{0}}\sim x_{i_{1}}\sim\cdots \sim x_{i_{n_{i}}}=b_{i}.\end{gather*}

\subsection{The jet factorisation system and the jet part}

In this section we recall the def\/initions of the jet factorisation system and the jet part of a~groupoid.

An arrow $f\colon A \rightarrow B$ in $\mathcal{E}/M$ is \emph{jet-dense} if\/f for all $b\colon X \rightarrow B$ there exists a cover $(\iota_{i}\colon X_{i}\rightarrow X)_{i\in I}$ and elements $a_{i}\colon X_i \rightarrow A$ such that $f(a_{i})\approx b_{i}$. We have written $b_i$ for the restriction of~$b$ to~$X_i$. An arrow $g\colon A \rightarrow B$ in $\mathcal{E}/M$ is \emph{jet-closed} if\/f it is a monomorphism and for all $a\colon X \rightarrow A$ and $b\colon X \rightarrow B$ such that $ga\approx b$ there exists a cover $(\iota_{i}\colon X_{i}\rightarrow X)_{i\in I}$ and elements $c_{i}\colon X_i \rightarrow A$ such that $a_{i}\approx c_{i}$ and $gc_{i}=b_{i}$. We have written~$a_i$ and~$b_i$ for the restrictions of~$a$ and~$b$ respectively to~$X_i$.

In the case $M=1$ the right class of the jet factorisation system has been studied before. For instance it is the class of formal-etale maps in~I.17 of~\cite{MR2244115}. In fact in Section~1.2 of~\cite{Johnstone2014} is it called the class of formally-open morphisms.
The sense in which these maps are open is ref\/lected in the following corollary that follows immediately from the def\/inition of jet closed.

\begin{Corollary}\label{cor:open-inclusion-is-jet-closed} The inclusion of an open subset $U$ into a manifold $M$ is jet closed.
\end{Corollary}

Now we recall from \cite{burke-relative-to-local-models} the results about the jet factorisation system that we need in the rest of this paper.

\begin{Lemma}\label{lem:existence-of-w-factorisations}Let $h\colon A\rightarrow E$ be an arrow in $\mathcal{E}/M$. Then there exists a~jet closed arrow~$g$ and a~jet dense arrow~$f$ such that $h=gf$. The mediating object in the factorisation has the following description:
\begin{gather*}B=\{x\in E\colon \exists\, a\in A,\, ha\approx x\}\xrightarrow{g} E.\end{gather*}
\end{Lemma}
\begin{proof}See Lemma 3.23 in~\cite{burke-relative-to-local-models}.
\end{proof}

\begin{Theorem}
Let $\mathbb{G}$ be a groupoid in $\mathcal{E}$. Then the subobject
 \begin{gather*}(G_{\infty},s_{\infty})=\{g\in (G,s)\colon esg\approx g\}\end{gather*}
is closed under composition and hence defines a subgroupoid $\mathbb{G}_{\infty} \rightarrowtail \mathbb{G}$ called the \emph{jet part of~$\mathbb{G}$}.
\end{Theorem}
 \begin{proof}
 See Corollary 4.4 and Proposition~4.18 in~\cite{burke-relative-to-local-models}.
 \end{proof}

\begin{Proposition}Let $L_{\infty}$ be the class of jet dense arrows and $R_{\infty}$ the class of jet closed arrows. Then the pair $(L_{\infty},R_{\infty})$ defines a $(\mathcal{E}/M)$-factorisation system.
\end{Proposition}
 \begin{proof}See Section 3.2 in \cite{burke-relative-to-local-models}.
 \end{proof}

\begin{Proposition}\label{proposition:dense-stable-along-closed-when-symmetric}
Let $g$ be jet dense and $k$ be jet closed in ${\mathcal E}/M$. Suppose that the relation $\approx$ is symmetric on the object $E$ and that the square
\begin{equation*}\begin{tikzcd}
A \dar{f} \rar[rightarrowtail]{h} & B \dar{g} \\
C \rar[rightarrowtail]{k} & E
\end{tikzcd}\end{equation*}
is a pullback. Then $f$ is also jet dense.
\end{Proposition}
\begin{proof}
 See Proposition 3.27 in \cite{burke-relative-to-local-models}.
\end{proof}

\subsection{Neighbour relation is symmetric for Lie groupoids}\label{sec:symmetry-of-neigbour-relation}

One of the assumptions that is required to prove Lie's second theorem in \cite{burke-relative-to-local-models} involves the symmetry of the neighbour relation $\sim$ def\/ined in Section~\ref{sec:infinitesimal-neighbour-relation}. More precisely, if $\mathbb{G}$ is a groupoid in~$\mathcal{E}$ with arrow space~$G$ and source map~$s$ then we need to assume that~$\sim$ is symmetric on the object~$(G,s)$ in $\mathcal{E}/M$. In this section we justify this assumption by proving that if~$\mathbb{G}$ is a~Lie groupoid then the relation $\sim$ is symmetric on the object $(G,s)$ in $\mathcal{E}/M$.

So suppose that $a,b\in_X(G,s)$ in $\mathcal{E}/M$ and $a\sim b$. By def\/inition we have a cover $(X_i \rightarrow X)_i$ such that for all $i$ there exist $W_i\in {\rm Spec}({\rm Weil})$, $\phi_i\in_{X_i}(G,s)^{D_{W_i}}$ and $d_i\in_{X_i}D_{W_i}$ making\vspace{-3mm}
\begin{equation*}
\begin{tikzcd}
 X_i \dar[swap]{(1_{X_i},d_i)} \drar{b_i} {}\\
 X_i\times D_{W_i} \rar{\phi_i} & (G,s)\\
 X_i \uar{(1_{X_i},0)} \urar[swap]{a_i} & {}
\end{tikzcd}
\end{equation*}
commute where $a_i$ and $b_i$ are the restrictions of $a$ and $b$ to $X_i$. We need to show that $b\sim a$.

\begin{Definition}\label{definition:s-trivialisation}
Let $s\colon G\rightarrow M$ be an arrow in $Man$ and $x\in G$. Then a pair of open embeddings $(\alpha_{x},\beta_{x})$ is an \emph{$s$-trivialisation centred at $x$} if\/f\vspace{-3mm}
\begin{equation*}\begin{tikzcd}
C^{k+n} \dar[twoheadrightarrow]{\pi} \rar[rightarrowtail]{\alpha_{x}} & G \dar[twoheadrightarrow]{s} \\
C^{k} \rar[rightarrowtail]{\beta_{x}} & M
\end{tikzcd}\end{equation*}
commutes and $\alpha_{x}(0)=x$.
\end{Definition}

\begin{Lemma} There exists a cover of $\iota_x\colon (X_{i,x}\rightarrow X_i)$ such that $\iota_x a_i$ factors through an $s$-trivialisation $C^{n+k} \rightarrowtail G$ around $a_i(x)$.
\end{Lemma}
\begin{proof} Let $X_i = (B_i,\xi_i)$. Since $s$ is a submersion we can choose for each $x\in B_i$ an $s$-trivialisation $\nu_x\colon C^{n+k}\rightarrowtail G$ centred at $a_i(x)$. Write $U_x$ for the image of $\nu_x$. Then the family $(\iota_x\colon a_i^{-1}(U_x) \rightarrow B_i)_{x\in B_i}$ covers $B_i$ in $\mathcal{E}$ and for each $x\in B_i$ the arrow $\iota_x a_i$ factors through~$U_x$. This means that $(\iota_{x}\colon (a_i^{-1}(U_x),\xi_i) \rightarrow X_i))_{x\in B_i}$ is a covering family in $\mathcal{E}/M$ such that $\iota_x a_i$ factors through~$U_x$. So we choose $X_{i,x}=((a_i^{-1}(U_x),\xi_i)$.
\end{proof}

Now using the cover $(X_{i,x}\rightarrow X)_{i,x}$ we show that $b\sim a$.

\begin{Lemma}\label{lem:approx-is-symmetric} If $d_{i,x}$, $a_{i,x}$, $b_{i,x}$ and $\phi_{i,x}$ are the restrictions of $d$, $a$, $b$ and $\phi$ respectively to $X_{i,x}$ then the arrows $\psi_{i,x}\colon X_{i,x}\times D_{W_i}\rightarrow (U_{x},s)$ defined by
\begin{gather*}\psi_{i,x}(u,d)=a_{i,x}(u)+_{s}\phi_{i,x}(u,d_{i,x}(u))-_{s}\phi(u,d)\end{gather*}
exhibit $b\sim a$ where $+_{s}$ and $-_{s}$ denote the fibrewise addition and subtraction. $($I.e., addition in the last $n$ coordinates of the $s$-trivialisation.$)$ Hence the infinitesimal neighbourhood relation is symmetric for all Lie groupoids.
\end{Lemma}
\begin{proof}
By construction the diagrams\vspace{-3mm}
\begin{equation*}
\begin{tikzcd}
 X_{i,x}\dar[swap]{(1_{X_{i,x}},d_{i,x})} \drar{b_{i,x}} {}\\
 X_{i,x} \times D_{W_i} \rar{\phi_{i,x}} & (U_{x},s)\\
 X_{i,x}\uar{(1_{X_{i,x}},0)} \urar[swap]{a_{i,x}} & {}
\end{tikzcd}
\end{equation*}

\noindent
commute for all $x\in B_i$. First we check that $\psi_{i,x}$ factors through~$U_x$. This follows from the equality $\psi(u,0)=b_{i,x}(u)$ and the fact that the inclusion of~$U_x$ into~$G$ is jet closed. Second we check that $\psi_{i,x}$ def\/ines an arrow in the slice category. But this follows from the fact that the three terms $a_{i,x}(u)$, $\phi_{i,x}(u,d_{i,x}(u))$ and $\phi(u,d)$ have the same source and the addition def\/i\-ning~$\psi_{i,x}$ is carried out in the last~$n$ coordinates of the $s$-trivialisation. Finally since
\begin{gather*}\psi(u,0)=a_{i,x}(u)+\phi_{i,x}(u,d_{i,x}(u))-\phi_{i,x}(u,0)=b_{i,x}(u)\end{gather*}
and
\begin{gather*}\psi(u,d_i(u))=a_{i,x}(u)+\phi_{i,x}(u,d_{i,x}(u))-\phi_{i,x}(u,d_{i,x}(u))=a_{i,x}(u)\end{gather*}
we conclude that $b\sim a$.
\end{proof}

\subsection{A trivialisation cover of the identity elements}\label{sec:trivialisation-of-identities}

In this section we construct a cover $(\phi_{em}\colon C^{n+k} \rightarrow G)_{m\in M}$ of $e(M)$ in $G$ with the property that each $\phi_{em}$ has a lift $\psi_{em}$ making\vspace{-3mm}
\begin{equation*}
 \begin{tikzcd}
 {} & \mathbb{G}^{\mathbb{I}} \dar{\mathbb{G}^{l}}\\
 C^{n+k} \urar[dashed]{\psi_{em}} \rar{\phi_{em}} & G
 \end{tikzcd}
\end{equation*}
commute and furthermore when we restrict $\psi_{em}$ to $e(M)$ the f\/illers we obtain are the constant f\/illers.
First we choose an $s$-trivialisation at $em$ such that the identity inclusion induces a section of the projection onto the f\/irst $k$ coordinates in the trivialisation.

\begin{Lemma}\label{lemma:trivialisation-with-identities-included}
If $m\in M$ then there is an $s$-trivialisation $(\alpha_{em},\beta_{em})$ at $em$ such that $e \beta_{em}$ factors through $\alpha_{em}$.
\end{Lemma}
\begin{proof}Let $(\alpha,\beta)$ be any $s$-trivialisation at $em$. Then if $\nu$ and $\xi$ are def\/ined in the pullback\vspace{-3mm}
\begin{equation*}\begin{tikzcd}
P \dar[rightarrowtail]{\nu} \rar[rightarrowtail]{\xi} & C^{k+n}\dar[rightarrowtail]{\alpha} \\
C^{k} \rar[rightarrowtail]{e\beta} & G
\end{tikzcd}\end{equation*}
then $\beta\pi\xi=s\alpha\xi=se\beta\nu=\beta\nu$ and so $\pi\xi=\nu$ because $\beta$ is a monomorphism. Now $P$ is an open set of $C^{k}$ and $0\in P$ because $e\beta(0)=\alpha(0)$. Since the derivative of $\nu$ has full rank at $0$ we can f\/ind an open embedding $\iota\colon C^{k}\rightarrowtail P$ such that $\nu\iota(0)=0$. Now let $\mu$ be def\/ined by the pullback\vspace{-3mm}
\begin{equation*}\begin{tikzcd}
C^{k+n} \dar[twoheadrightarrow]{\pi} \rar[rightarrowtail]{\mu} & C^{k+n}\dar[twoheadrightarrow]{\pi} \\
C^{k}\rar[rightarrowtail]{\nu\iota} \uar[bend left,dashed,rightarrowtail]{\rho} & C^{k}
\end{tikzcd}\end{equation*}
and $\rho$ be induced by the pair $(1_{P},\xi\iota)$. Then $e\beta\nu\iota=\alpha\xi\iota=\alpha\mu\rho$ and the $s$-trivialisation that we require is $(\alpha_{em},\beta_{em})=(\alpha\mu,\beta\nu\iota)$.
\end{proof}

This means that for each $\vec{x}\in C^k$ the arrow $\psi(\rho(\vec{x}),\vec{y})$ is an identity arrow. The $\phi_{em}$ that we require will be the $\alpha_{em}$ obtained in Lemma~\ref{lemma:trivialisation-with-identities-included}. Now we can construct a lift $\psi_{em}\colon C^{k+n} \rightarrow \mathbb{G}^{\mathbb{I}}$ for $\phi_{em}$ as follows. For each $(\vec{x},\vec{y})\in C^{k+n}$ we have a source constant path
\begin{gather}\label{source-constant-path}
 a \mapsto (\vec{x}, a\vec{y}+(1-a)\rho(\vec{x})),
\end{gather}
which starts at an identity. Since~\eqref{source-constant-path} is smooth in $\vec{x}$ and $\vec{y}$ it induces an arrow $\psi_{em}\colon C^{k+n} \rightarrow \mathbb{G}^{\mathbb{I}}$. Moreover by construction the restriction of $\psi_{em}$ to $e(M)$ are the constant paths at identity arrows.

\subsection{A cover of the jet part}\label{sec:cover-of-jet-part}

In Section~\ref{sec:trivialisation-of-identities} we constructed a cover $(\phi_{em}\colon C^{n+k}\rightarrow G)_{m\in M}$ of $e(M)$ in $G$ satisfying certain properties on restriction to~$e(M)$. In this section we show that the $\phi_{em}$ also induce a cover of the object $(G_{\infty},s_{\infty})$ in $\mathcal{E}/M$.

\begin{Lemma}\label{lemma:Lifting Cover to Ginfty}
There is an inclusion $j\colon (G_{\infty},s_{\infty})\rightarrowtail \bigcup_{m}(U_{m},s\phi_{em})$ such that $\bigcup_m \phi_{em} \circ j=\iota_{G}^{\infty}$.
\end{Lemma}
\begin{proof}
By hypothesis we have an inclusion $(M,1_M)\rightarrowtail \bigcup_{m}(U_{m},s\phi_{em})$ such that $\iota\circ m=e$. Since the inclusion $\iota$ is jet closed in $\mathcal{E}/M$ the square
\begin{equation*}\begin{tikzcd}
(M,1) \dar[rightarrowtail]{e_{\infty}} \rar[rightarrowtail]{m} & (\bigcup_{m}U_{m},s\phi_{em}) \dar[rightarrowtail]{\bigcup_m \phi_{em}} \\
(G_{\infty},s_{\infty}) \urar[rightarrowtail,dashed]{\exists !\, j} \rar[rightarrowtail]{\iota_{G}^{\infty}} & (G,s)
\end{tikzcd}\end{equation*}
has a unique (monic) f\/iller.
\end{proof}

\begin{Corollary}\label{corollary:Cover of Ginfty}
Let the objects $(V_{m},s_{\infty}\phi_{em})$ of $\mathcal{E}/M$ be defined by the pullbacks
\begin{equation*}\begin{tikzcd}
(V_{m},s_{\infty}\phi_{em}) \dar[rightarrowtail]{\chi_{m}} \rar[rightarrowtail]{u_m} & (U_{m},s\phi_{em}) \dar[rightarrowtail]{\phi_{em}} \\
(G_{\infty},s_{\infty}) \rar[rightarrowtail]{\iota_G^{\infty}} & (G,s)
\end{tikzcd}\end{equation*}
then because colimits are stable under pullback the bottom right square in
\begin{equation*}\begin{tikzcd}
(G_{\infty},s_{\infty})\arrow[bend left,rightarrowtail]{drr}{j} \arrow[bend right]{ddr}[swap]{1_{G_{\infty}}} \arrow[dashed]{dr}{\eta}& {} & {}\\
{} & (\bigcup_{m}V_{m},s_{\infty}\phi_{em}) \rar[rightarrowtail]{\bigcup_{m}u_m} \dar[rightarrowtail]{\bigcup_m \chi_{m}} & (\bigcup_{m}U_{m},s\phi_{em}) \dar[rightarrowtail]{\bigcup_m \phi_{em}}\\
{} & (G_{\infty},s_{\infty}) \rar[rightarrowtail]{\iota_G^{\infty}} & (G,s)
\end{tikzcd}\end{equation*}
is a pullback. But then the arrow $\eta$ induced by the pair $(1_{G_{\infty}},j)$ is an isomorphism and hence $\bigcup_{m\in M}\chi_{m}$ is a cover of $(G_{\infty},s_{\infty})$.
\end{Corollary}

\subsection{Jet part of a Lie groupoid is internal path connected}\label{sec:jet-part-path-connected}

Now we combine Section~\ref{sec:trivialisation-of-identities} and Section~\ref{sec:cover-of-jet-part} to show that the jet part of a Lie groupoid is $\mathcal{E}$-path connected. It will suf\/f\/ice to show that when we restrict the f\/illers $\psi_{em}\colon C^{n+k}\rightarrow \mathbb{G}^{\mathbb{I}}$ def\/ined in Section~\ref{sec:trivialisation-of-identities} along $u_m$ we get an arrow that factors through~$(\mathbb{G}_{\infty}^{\mathbb{I}},s_{\infty})$. Then $\psi_{em}u_m$ is a f\/iller for~$\chi_m$.

So let $V_{m}$ and $W_{m}$ be def\/ined by the iterated pullback:
\begin{equation*}\begin{tikzcd}
(W_{m},\iota) \dar \rar[rightarrowtail]{v_m} & (V_{m},s_{\infty}\chi_{em}) \dar{\chi_{em}} \rar[rightarrowtail]{u_m} & (U_{m},s\phi_{em})\dar{\phi_{em}} \\
(M,1) \rar[rightarrowtail]{e_{\infty}} & (G_{\infty},s_{\infty}) \rar[rightarrowtail]{\iota^{\infty}_{G}} & (G,s)
\end{tikzcd}\end{equation*}
and note that the $\chi_{em}$ are Penon open because the $\phi_{em}$ are. Then by Proposition~\ref{proposition:dense-stable-along-closed-when-symmetric} and Lemma~\ref{lem:approx-is-symmetric} we deduce that $v_m$ is jet dense. Since we have chosen $\psi_{em}$ such that $\psi_{em} v_m u_m$ are the constant functions $c_m$ the square
\begin{equation*}\begin{tikzcd}
(W_m,\iota) \rar[rightarrowtail]{c_m} \dar[rightarrowtail]{v_m} & \big(\mathbb{G}_{\infty}^{\mathbb{I}},s_{\infty}\big) \dar{({\iota_{\mathbb{G}}^{\infty}})^{\mathbb{I}}}\\
(V_m,s) \urar[dashed][description]{\zeta_m} \rar{\psi_{em} u_m} & \big(\mathbb{G}^{\mathbb{I}},s\big)
\end{tikzcd}\end{equation*}
commutes and has a unique f\/iller. This means that the $\phi_{em}$ form a Penon open cover of $G_{\infty}$ whose f\/illers factor through $\mathbb{G}_{\infty}^{\mathbb{I}}$. By pulling back this cover along generalised elements $X \rightarrow G_{\infty}$ we deduce that the jet part $\mathbb{G}_{\infty}$ is $\mathcal{E}$-path connected.

\section{Integral completeness}\label{sec:integral-completeness}

One of the main assumptions that we require to prove Lie's second theorem in \cite{burke-relative-to-local-models} is that of integral completeness.
Recall from Def\/inition~\ref{def:integral-complete} that an arbitrary groupoid $\mathbb{G}$ in a well-adapted model $\mathcal{E}$ of synthetic dif\/ferential geometry is integral complete if\/f
\begin{gather*}\mathbb{G}^{\mathbb{I}} \xrightarrow{\mathbb{G}^{\iota_{\mathbb{I}}^{\infty}}} \mathbb{G}^{\mathbb{I}_{\infty}}\end{gather*}
is an isomorphism in ${\rm Gpd}(\mathcal{E})$ where $\mathbb{I}$ is the pair groupoid on the unit interval~$I$. In Section~\ref{sec:paths-of-infinitesimals} we show that the classical $A$-paths (see for instance~\cite{2000JACKolkandJJDuistermaat44}) correspond to global sections of $\mathbb{G}^{\mathbb{I}_{\infty}}$ in~$\mathcal{E}$ and the classical $G$-paths (see also~\cite{2000JACKolkandJJDuistermaat44}) correspond to global sections of $\mathbb{G}^{\mathbb{I}}$ in $\mathcal{E}$. In Section~\ref{sec:Lie Groupoids are Integral Complete} we show that all classical Lie groupoids are integral complete. But f\/irst we give a more explicit description of the arrow space of~$\mathbb{I}_{\infty}$.

\subsection{Representing object for inf\/initesimal paths is trivial}\label{sec:arrow-space-of-A-paths}

In this section we show that the arrow space $\mathbb{I}_{\infty}^{\mathbf 2}$ of $\mathbb{I}_{\infty}$ is isomorphic to $I\times D_{\infty}$.

Recall from Lemma~\ref{lem:existence-of-w-factorisations} that the arrow space of $\mathbb{I}_{\infty}$ is characterised as follows. A generalised element $(a,b)\in (I^2,\pi_1)$ is in $(\mathbb{I}_{\infty}^{\mathbf 2},\pi_1)$ if\/f there exists $m\in (I,1_I)$ such that $(m,m)\approx (a,b)$. By def\/inition of $\approx$ if $b-a\in D_{\infty}$ then $a\approx b$. This means that it will suf\/f\/ice to prove the following result:
\begin{Lemma}\label{lem:b-a-nilpotent}
 If $(a,b)\colon X \rightarrow I^2$ and $a\approx b$ in $\mathcal{E}$ then $b-a\in D_{\infty}$.
\end{Lemma}
\begin{proof}
First suppose that $a\sim b$. This means that there exist $W\in {\rm Spec}({\rm Weil})$, $\phi\in I^{D_W}$ and $d\in D_W$ such that $\phi(0)=a$ and $\phi(d)=b$. Then by the Kock--Lawvere axiom $b=a+N$ for some nilpotent $N$.

Suppose now that $a\approx b$. This means that there exist $a_0,\dots ,a_n$ such that $a=a_0\sim a_1\sim \dots \sim a_n=b$. Now we know that for all $i\in\{1,\dots ,n\}$ there exists $k_i\in\mathbb{N}$ such that $(a_i-a_{i-1})^{k_i}=0$. But then $(b-a)^{\Sigma_i k_i}=0$ as required.
\end{proof}

\begin{Corollary}\label{lem:A-path-is-trivial}
 The groupoid $\mathbb{I}_{\infty}$ has underlying reflexive graph isomorphic to
 \begin{equation*}\begin{tikzcd}
 I\times D_{\infty} \arrow[yshift=1.5ex]{r}{+} \rar[leftarrow][description]{e} \arrow[yshift=-1.5ex]{r}[swap]{\pi_{1}} & I,
 \end{tikzcd}\end{equation*}
 where $e=(1_{I},0)$ and composition $I\times D_{\infty}\times D_{\infty}\rightarrow I\times D_{\infty}$ defined by
 \begin{gather*}(a,d,d')\mapsto (a,d+d').\end{gather*}
\end{Corollary}

 \begin{proof} For all $a\in I$ we have $a\approx a+d$ and we can def\/ine an arrow $I\times D_{\infty} \rightarrow \mathbb{I}_{\infty}^{\mathbf 2}$ by $(a,d)\mapsto (a,a+d)$. The inverse $(a,b)\mapsto (a,b-a)$ factors through $I\times D_{\infty}$ by Lemma~\ref{lem:b-a-nilpotent}.
 \end{proof}

\subsection{Formal group laws}\label{sec:formal-group-laws}

When we form the inf\/initesimal part of a category in \cite{burke-relative-to-local-models} our construction corresponds to the part of a Lie group represented by its formal group law. Following~\cite{MR506881} we def\/ine an $n$-dimensional formal group law $F$ to be an $n$-tuple of power series in the variables $X_{1},\dots ,X_{n}$; $Y_{1},\dots ,Y_{n}$ with coef\/f\/icients in $\mathbb{R}$ such that the equalities
\begin{gather}
F\big(\vec{X},\vec{0}\big)=\vec{X},\qquad F\big(\vec{0},\vec{Y}\big)=\vec{Y}\qquad \text{and} \qquad F\big(F\big(\vec{X},\vec{Y}\big),\vec{Z}\big)=F\big(\vec{X},F\big(\vec{Y},\vec{Z}\big)\big)
\end{gather}
hold. We refer to the Introduction of~\cite{MR506881} for the construction of a formal group law from a~Lie group. In fact the category of Lie algebras and formal group laws are shown to be equivalent in Theorem~3 of Section~V.6 of Part~2 in~\cite{MR2179691}.

In the following example we show how to reformulate the construction of a formal group law from a Lie group in terms of the inf\/initesimal elements of the Lie group.
\begin{Example}\label{example-04/27/1504:43:57 PM} Let $(G,\mu)$ be a Lie group whose underlying smooth manifold is $n$-dimensional. Since $G$ is locally isomorphic to $R^{n}$ we see that its jet part is a group of the form $(D_{\infty}^{n},\mu)$ by a~straightforward extension of Lemma~\ref{lem:b-a-nilpotent}. Now to give a multiplication
\begin{gather*}\mu\colon \ D_{\infty}^{n}\times D_{\infty}^{n}\rightarrow D_{\infty}^{n}\end{gather*}
is to give arrows
\begin{gather*}f_{1},\dots ,f_{n}\colon \ (D_{\infty})^{2n}\rightarrow R\end{gather*}
taking values in nilpotent elements. Now we have that
\begin{gather*}(D_{\infty})^{2n}=\bigcup_{k}(D_{k})^{2n}\end{gather*}
and so, since $\mathcal{E}(-,R)$ sends colimits to limits the hom-set $\mathcal{E}(D_{\infty}^{2n},R)$ is given by the limit
\begin{gather*}\cdots \rightarrow\mathcal{E}\big(D_{k+1}^{2n},R\big)\rightarrow\mathcal{E}\big(D_{k}^{2n},R\big)\rightarrow \cdots, \end{gather*}
which by the Kock--Lawvere axiom is equivalently the limit of the polynomial algebras
\begin{gather*}\cdots \rightarrow \mathbb{R}[X_{1},\dots ,X_{2n}]/I_{k+1}\rightarrow \mathbb{R}[X_{1},\dots ,X_{2n}]/I_{k}\rightarrow \cdots, \end{gather*}
where $I_{k}$ is the ideal generated by $(X_{1}^{k},X_{2}^{k},\dots ,X_{2n}^{k})$. This means that $\mathcal{E}(D_{\infty}^{2n},R)$ can be identif\/ied with the ring $\mathbb{R}[[X_{1},\dots ,X_{2n}]]$ of formal power series. Now the condition that the~$f_{i}$ take values in the nilpotent elements implies that the constant term of the power series~$p_{i}$ corresponding to~$f_{i}$ is zero. Under this correspondence, the group axioms for $G$ correspond to the axioms making $p_{1},\dots ,p_{n}$ into a formal group law.
\end{Example}

\subsection{Paths of inf\/initesimals}\label{sec:paths-of-infinitesimals}

The correct notion of a path of inf\/initesimal arrows in a Lie groupoid $\mathbb{G}$ is that of an $A$-path (see for instance~\cite{MR1973056}). In the topos $\mathcal{E}$ the \emph{object of $A$-paths $A(\mathbb{G})$ associated to $\mathbb{G}$} is the subobject of all $\phi\in G^{I\times D}$ such that for all $a\in I$ and all $d\in D$ the arrows $\phi(a,0)$ are identity arrows, the $\phi(a,-)$ are source constant and $t\phi(a,d)=t\phi(a+d,0)$. Note that since $G^D\cong TG$ the global sections of~$A(\mathbb{G})$ are precisely the $A$-paths def\/ined in Section~1 of~\cite{MR1973056}.

In this section we show that $A(\mathbb{G})\cong \mathbb{G}^{\mathbb{I}_{\infty}}$ in $\mathcal{E}$ where~$\mathbb{I}_{\infty}$ is the jet part of the pair groupoid~$\mathbb{I}$ on the unit interval~$I$. Using Corollary~\ref{lem:A-path-is-trivial} we see that $\mathbb{G}^{\mathbb{I}_{\infty}}$ is the subobject of all $\phi\in G^{I\times D_{\infty}}$ such that for all $a\in I$ and all $d\in D_{\infty}$ the arrows $\phi(a,0)$ are identity arrows, the~$\phi(a,-)$ are source constant and not only does $t\phi(a,d)=t\phi(a+d,0)$ hold but indeed
\begin{gather*}\phi(a,d+d')=\phi(a+d,d')\phi(a,d)\end{gather*}
holds for all $d,d'\in D_{\infty}$. This means that there is a natural restriction arrow $\mathbb{G}^{\mathbb{I}_{\infty}} \rightarrow A(\mathbb{G})$. In this section we describe its inverse.

To do this we def\/ine an arrow $v\colon G^{I\times D} \rightarrow G^{I\times D_{\infty}}$ which satisf\/ies $v(\phi)(a,d+d')=v(\phi)(a+d,d')v(\phi)(a,d)$ for all $d,d'\in D_{\infty}$. Recall that $D_{\infty}=\bigcup_i D_i$ and so it will suf\/f\/ice to f\/ind for all $i\in\mathbb{N}$ an arrow $v_i\colon G^{I\times D} \rightarrow G^{I\times D_i}$ such that $v_{i+j}(\phi)(a,d+d')=v_j(\phi)(a+d,d')v_i(\phi)(a,d)$ for all $d\in D_i$ and $d'\in D_j$.

Now we recall the following slight generalisation of the \emph{Bunge axiom} that is Proposition~4 in Section~2.3.2 in \cite{MR1385464}:

\begin{Lemma}\label{lem:bunge-axiom}
Let $i\in \mathbb{N}$ and consider the arrows $f_{1},\dots ,f_{i}\colon D^{i-1}\rightarrow D^{i}$ defined by
\begin{gather*}f_{m}(d_{1},\dots ,d_{i-1})=(d_{1},\dots ,d_{m-1},0,d_{m},\dots ,d_{i-1}).\end{gather*}
Then for any microlinear space $G$ the arrow
\begin{gather*}G^{D_{i}}\xrightarrow{G^{+}}G^{D^{i}}\end{gather*}
is the joint equaliser of $G^{f_{1}},\dots,G^{f_{i}}$.
\end{Lemma}

Using Lemma~\ref{lem:bunge-axiom} we see that it will now suf\/f\/ice to f\/ind for all $i\in \mathbb{N}$ an arrow $v_i\colon G^{I\times D} \rightarrow G^{I\times D^i}$ such that for all $m,l\in \{1,\dots ,i\}$ the equalities $G^{I\times f_m}v_i(\phi)=G^{I\times f_l}v_i(\phi)$ and \begin{gather*}
v_{i+j}(\phi)(a,(d_1,\dots,d_{i+j}))=v_j(\phi)\left(a+\sum_{m=1}^{i} d_m,(d_{i+1},\dots,d_{i+j})\right)v_i(\phi)(a,(d_1,\dots,d_i))
\end{gather*} hold for the $f_i$ def\/ined in Lemma~\ref{lem:bunge-axiom}.

\begin{Lemma} The restriction $\mathbb{G}^{\mathbb{I}_{\infty}} \rightarrow A(\mathbb{G})$ has an inverse.
\end{Lemma}
 \begin{proof}
The arrows $v_i\colon G^{I\times D} \rightarrow G^{I\times D^i}$ def\/ined by
 \begin{gather*}v_i(\phi)(a,(d_1,\dots,d_i))=\phi(a+\Sigma_{m=1}^{i-1} d_m,d_i)\cdots \phi(a+d_1,d_2)\phi(a,d_1) \end{gather*}
satisfy $G^{I\times f_m}v_i(\phi)=G^{I\times f_l}v_i(\phi)$ because $\phi(a,0)$ are identity arrows for all $a\in I$ and satisf\/ies
\begin{gather*}
v_{i+j}(\phi)(a,(d_1,\dots,d_{i+j}))=v_j(\phi)\left(a+\sum_{m=1}^{i} d_m,(d_{i+1},\dots,d_{i+j})\right)v_i(\phi)(a,(d_1,\dots,d_i))
\end{gather*} by construction. It is easy to see that the $v_i$ def\/ine an inverse to the restriction.
 \end{proof}

\subsection{Integration of paths of inf\/initesimals is groupoid enriched}

Recall that in Def\/inition~\ref{def:integral-complete} we def\/ined the notion of integral complete groupoid using an isomorphism in the category ${\rm Gpd}(\mathcal{E})$. The following result show that we only need to check this condition on the space of objects which is an object of~$\mathcal{E}$.

\begin{Proposition}\label{prop:groupoid-enriched-integration}
If $\mathbb{G}^{\iota_{\infty}}\colon \mathbb{G}^{\mathbb{I}}\rightarrow\mathbb{G}^{\mathbb{I}_{\infty}}$ is an isomorphism in a well-adapted model $\mathcal{E}$ then it is an isomorphism of groupoids also.
\end{Proposition}
\begin{proof}
We need to show that natural transformations extend uniquely, i.e.,
\begin{equation*}\begin{tikzcd}
\mathbb{I}_{\infty}\times\pmb{2} \rar{\forall\,\Phi} \dar[swap]{\iota} & \mathbb{G} \\
\mathbb{I}\times\pmb{2}. \urar[dashed][swap]{\exists!\,\Psi} &
\end{tikzcd}
\end{equation*}
Let $\psi_{0}$, $\psi_{1}$ be the unique lifts of $\phi$ precomposed with the two inclusions of $1$ into $\pmb{2}$. If for all $x\rightarrow y$ in $\mathbb{I}$ the diagram
\begin{equation}\label{PsiDef}
	\begin{tikzcd}[column sep=4cm]
\Phi(x,1) \rar{\psi_{1}(x\rightarrow y)} \dar[leftarrow,swap]{\Phi(x,0\rightarrow 1)} & \Phi(y,1) \dar[leftarrow]{\Phi(y,0\rightarrow 1)} \\
\Phi(x,0) \rar{\psi_{0}(x\rightarrow y)} & \Phi(y,0)
\end{tikzcd}
\end{equation}
commutes then we can def\/ine $\Psi(x\rightarrow y,0\rightarrow 1)$ to be this common value. To this end def\/ine $\theta\colon \mathbb{I}\rightarrow \mathbb{G}$ to take $x\rightarrow y$ to
\begin{equation*}\begin{tikzcd}[column sep=4cm]
\Phi(x,1) \dar[swap]{\Phi(x,1\rightarrow 0)} & \Phi(y,1) \dar[leftarrow]{\Phi(y,0\rightarrow 1)} \\
\Phi(x,0) \rar{\psi_{0}(x\rightarrow y)} & \Phi(y,0),
\end{tikzcd}
\end{equation*}
when we restrict to $\mathbb{I}_{\infty}$ (i.e., take $y=x+d$) we see that
\begin{equation*}\begin{tikzcd}[column sep=1.7cm]
\Phi(x,1) \dar[swap]{\Phi(x,1\rightarrow 0)} & \Phi(x+d,1) \dar[swap,leftarrow]{\Phi(x+d,0\rightarrow 1)} \\
\Phi(x,0) \rar{\Phi(x\rightarrow x+d,0)} & \Phi(x+d,0)
\end{tikzcd}
=
\begin{tikzcd}[column sep=1.7cm]
\Phi(x,1) \rar{\Phi(x\rightarrow x+d,1)} & \Phi(x+d,1) \\
{} & {}
\end{tikzcd}
\end{equation*}
and so by the uniqueness of lifts $\theta=\psi_{1}$ and \eqref{PsiDef} commutes.
\end{proof}

\subsection{Lie groupoids are integral complete}\label{sec:Lie Groupoids are Integral Complete}

We show that $\mathbb{G}^{\iota_{\mathbb{I}}^{\infty}}\colon \mathbb{G}^{\mathbb{I}} \rightarrow\mathbb{G}^{\mathbb{I}_{\infty}}$ is an isomorphism in ${\rm Gpd}(\mathcal{E})$. By Proposition~\ref{prop:groupoid-enriched-integration} it will suf\/f\/ice to show that $\mathbb{G}^{\iota_{\mathbb{I}}^{\infty}}$ is an isomorphism in $\mathcal{E}$. More concretely, we show that for all representable objects $X$ and arrows $\phi\colon X \rightarrow \mathbb{G}^{\mathbb{I}_{\infty}}$ there exists a~(unique) $\psi\colon X \rightarrow \mathbb{G}^{\mathbb{I}}$ such that $\mathbb{G}^{\iota_{\mathbb{I}}^{\infty}}\psi=\phi$.

By Corollary~\ref{lem:A-path-is-trivial} arrows $\phi\colon X \rightarrow \mathbb{G}^{\mathbb{I}_{\infty}}$ correspond to arrows $\phi\colon X\times I\times D_{\infty} \rightarrow G$ such that $\phi(x,a,0)$ are identity arrows and $\phi(x,a,-)$ are source constant. It is easy to see that arrows $\psi\colon X \rightarrow \mathbb{G}^{\mathbb{I}}$ correspond to arrows $\psi\colon X\times I \rightarrow G$ such that~$\psi(x,0)$ are identity arrows and~$\psi(x,-)$ are source constant. At this point it is convenient to assume that the topos~$\mathcal{E}$ is generated by a~subcanonical site whose underlying category is a full subcategory of the category of af\/f\/ine $C^{\infty}$-schemes as def\/ined in Def\/inition~\ref{def:smooth-affine-schemes}. In particular this means that every representable object is a~closed subset of~$R^n$ for some $n\in \mathbb{N}$. Recall from Lemma~2.26 in~\cite{MR2954043} that if we are given a smooth function that has as domain any closed subset of~$\mathbb{R}^n$ we can lift it to a~smooth function on the whole of~$\mathbb{R}^n$. Therefore since every representable $X$ is a~closed subset of~$R^n$ for some $n\in \mathbb{N}$ it will suf\/f\/ice to prove the result in the case $X=R^n$.

\begin{Theorem}For all $\phi\colon R^n\times I \times D_{\infty} \rightarrow G$ such that $\phi(x,a,0)$ are identity arrows, $\phi(x,a,-)$
is source constant and $\phi(a,d+d')=\phi(a+d,d')\phi(a,d)$ for $x\in R^n$, $a\in I$ and $d,d'\in D_{\infty}$ there exists a unique $\psi\colon R^n\times I \rightarrow G$ such that $\psi(x,0)$ are identity arrows, $\psi(x,-)$ is source constant and $\psi(x,a+d)=\phi(x,a,d)\psi(x,a)$ for all $d\in D_{\infty}$.
\end{Theorem}
\begin{proof}
To do this we make rigorous the intuitive idea of composing together inf\/initely many inf\/initesimal arrows to get a macroscopic arrow. First let $\phi_0=s\phi(-,-,0)=t\phi(-,-,0)$. Then the pullback
\begin{equation*}
\begin{tikzcd}
\big(R^n\times I\big)\tensor[_{\phi_0}]{\times}{_t} G \rar \dar[twoheadrightarrow]{\pi_{0,1}} & G \dar[twoheadrightarrow]{t},\\
R^n\times I \rar{\phi_0} & M
\end{tikzcd}
\end{equation*}
is a manifold because~$t$ is a submersion. Since $\phi(x,a,-)$ is source constant the inf\/initesimal action
\begin{align*}
R^n\times D\times \left(\big(R^n\times I\big)\tensor[_{\phi_0}]{\times}{_t} G\right) & \rightarrow \left(\big(R^n\times I\big)\tensor[_{\phi_0}]{\times}{_t} G\right),\\
(y,d,x,a,g)&\mapsto (x,a+d,\phi(x,a,d)\circ g)
\end{align*}
def\/ines a collection of smoothly parameterised vector f\/ields on $(R^n\times I)\tensor[_{\phi_0}]{\times}{_t} G$. Since the solution curves of smooth vector f\/ields have a smooth dependence on parameters we obtain using the initial conditions $\psi(y,0)=(y,0,\phi(y,0,0))$ a parameterised solution $\psi\colon R^n \times I\rightarrow (R^n\times I)\tensor[_{\phi_0}]{\times}{_t} G$ which satisf\/ies:
\begin{itemize}\itemsep=0pt
\item $\psi(y,0)=(\psi_1(y,0),\psi_2(y,0),\psi_3(y,0))=(y,0,\phi(y,0,0))$,
\item $\psi_1(y,a+d)=\psi_1(y,a)$,
\item $\psi_2(y,a+d)=\psi_2(y,a)+d$,
\item $\psi_3(y,a+d)=\phi(\psi_1(y,a),\psi_2(y,a),d)\circ\psi_3(y,a)$.
\end{itemize}
Therefore
\begin{itemize}\itemsep=0pt
\item $\psi_1(y,a)=\psi_1(y,0)=y$,
\item $\psi_2(y,a)=\psi_2(y,0)+a=a$,
\item $\psi_3(y,a+d)=\phi(\psi_1(y,a),\psi_2(y,a),d)\psi_3(y,a)=\phi(y,a,d)\circ\psi_3(y,a))$,
\end{itemize}
and $\psi_3$ is the map we require. Now we check that $\psi_3$ does indeed def\/ine a $G$-path. The map $\psi_3$ is source constant in the second variable because
\begin{gather*}s\psi_3(y,a+d)=s\left(\phi(y,a,d)\circ\psi(x,a)\right)=s\psi_3(y,a)\end{gather*}
for all $d\in D$. Finally we appeal to Proposition~2.7 in~\cite{MR2213027} to conclude that $\psi_3(y,a+d)=\phi(y,a,d)\psi_3(y,a)$ holds for all $d\in D_{\infty}$.
\end{proof}

\subsection*{Acknowledgements}

The author is very grateful for the constructive comments of\/fered by and the important corrections indicated by the editor and referees. The author would like to acknowledge the assistance of Richard Garner, my Ph.D.\ supervisor at Macquarie University Sydney, who provided valuable comments and insightful discussions in the genesis of this work. In addition the author is grateful for the support of an International Macquarie University Research Excellence Scholarship.

\pdfbookmark[1]{References}{ref}
\LastPageEnding

\end{document}